\input ./style/arxiv-general.cfg
\documentclass[MSNbibl,number,citesort,seceqn,dvips]{arxbj}
\makeatletter
   \@ifpackageloaded{graphicx}{}{\usepackage{graphicx}}
\makeatother
\usepackage{mathbh,upgreek}

\aid{0}
\volume{22}
\issue{1}
\pubyear{2016}
\firstpage{242}
\lastpage{274}
\doi{10.3150/14-BEJ657} 
\docsubty{FLA}

\makeatletter

\newcommand{\lleft}{\left}
\newcommand{\rrvert}{\vert}
\newcommand{\rright}{\right}
\newcommand{\llvert}{\vert}

\newtheorem{them}{Theorem}[section]
\newtheorem{prop}{Proposition}[section]
\newtheorem{cor}{Corollary}[section]
\newtheorem{lem}{Lemma}[section]
\newremark{Comments}{Comments}
\newremark{Remark}{Remark}

\def\F{\mathcal{F}}
\def\K{\mathcal{K}}
\def\R{\mathbb{R}}
\def\P{\mathbb{P}}
\def\E{\mathbb{E}}
\newcommand{\Var}{\operatorname{Var}}

\newcommand{\eqref}[1]{(\ref{#1})}
\newcommand{\mathds}{\mathbh}

\renewcommand{\pi}{\uppi}
\renewcommand{\emptyset}{\varnothing}

\def\sfrac#1#2{#1/#2}
\def\vfrac#1#2{(#1)/#2}
\def\afrac#1#2{#1/(#2)}

\def\sklfrac#1#2{(#1/#2)}

\renewcommand{\widehat}{\hat}
\makeatother

\begin{document}
\begin{frontmatter}

\title{Non-asymptotic detection of two-component mixtures with unknown means}
\runtitle{Non-asymptotic detection of mixtures with unknown means}

\begin{aug}
\author[A]{\inits{B.}\fnms{B\'eatrice} \snm{Laurent}\thanksref{e1}\ead[label=e1,mark]{beatrice.laurent@insa-toulouse.fr}},
\author[A]{\inits{C.}\fnms{Cl\'ement}~\snm{Marteau}\thanksref{e2}\ead[label=e2,mark]{clement.marteau@insa-toulouse.fr}} \and\\ 
\author[A]{\inits{C.}\fnms{Cathy}~\snm{Maugis-Rabusseau}\corref{}\thanksref{e3}\ead[label=e3,mark]{cathy.maugis@insa-toulouse.fr}}
\address[A]{Institut de Math\'ematiques de Toulouse, INSA de Toulouse,
Universit\'e de Toulouse
INSA de Toulouse, 135 avenue de Rangueil, 31077 Toulouse Cedex 4, France.\\
\printead{e1}; \printead*{e2};\\ \printead*{e3}}
\end{aug}

\received{\smonth{4} \syear{2013}}
\revised{\smonth{2} \syear{2014}}

%
\begin{abstract}
This work is concerned with the detection of a mixture distribution
from a $\mathbb{R}$-valued sample.
Given a sample $X_1,\dots, X_n$ and an even density $\phi$, our aim
is to detect whether the sample distribution is $\phi(\cdot-\mu)$
for some unknown mean $\mu$, or is defined as a two-component mixture
based on translations of $\phi$.
We propose a procedure
which is based on several spacings of the order statistics, which
provides a level-$\alpha$ test for all $n$. Our test is therefore a multiple
testing procedure and we prove from a theoretical and practical
point of view that it automatically adapts to the proportion of the
mixture and to the difference of the means of the two components of the mixture
under the alternative. From a theoretical point of view, we prove the
optimality of the power of our procedure in various situations.
A simulation study shows the good performances of our test compared
with several classical procedures.
\end{abstract}

%
\end{frontmatter}
%
\section{Introduction}
In this paper, the detection problem of a mixture distribution from a
$\mathbb{R}$-valued sample is considered. Let $(X_1,\dots,X_n)$ be i.i.d.
random variables from an unknown distribution $F$. All along the paper,
$F$ is assumed to admit a density $f$ w.r.t. the Lebesgue measure
on $\mathbb{R}$. The sample is said to be distributed from a mixture
when $f$ belongs to the set
%
\begin{equation}
\label{F1} \mathcal{F}_1 = \bigl\lbrace x\in\mathbb{R}\mapsto(1-
\varepsilon) \phi(x-\mu_1 ) + \varepsilon\phi(x - \mu_2);
\varepsilon\in\,]\,0,1[, (\mu_1,\mu_2)\in\R^2,
\mu_1<\mu_2 \bigr\rbrace,
\end{equation}
where $\phi(\cdot)$ denotes a density. In this paper, $\phi(\cdot)$ is
assumed to be an even known density, and when Gaussian mixtures are
considered, $\phi(\cdot)=\phi_G(\cdot)$ with
\[
\phi_G(x)=\frac{1}{\sqrt{2\pi}} \exp \biggl(-\frac{x^2}{2} \biggr),
\qquad \forall x\in\R. %
\]
For a complete introduction about mixtures, we refer to \cite
{McLachlan_Peel}. The two-component mixtures are often encountered in
practice, for instance, in biology and health science. They allow to
model situations where a population can be discriminated into two
different groups. The first subpopulation is then assumed to be
distributed following the density $\phi(\cdot-\mu_1)$ while the second
one follows the density $\phi(\cdot - \mu_2)$. The probability that an
observation $X_i$ arises from the first (resp. the second)
subpopulation is then modeled by $1-\varepsilon$ (resp. $\varepsilon$).

This model has been intensively studied and many paths have been
explored in order to provide a satisfying inference. In particular, the
detection problem has attracted a lot of attention in the last two
decades. The main goal is not to provide the best estimation of the
parameters of interest $(\varepsilon,\mu_1,\mu_2)$ but rather to decide
whether the incoming observations are following a mixture distribution
or not. In other words, one wants to detect if the sample of interest
comes from a homogeneous or heterogeneous population. Let $\mathcal
{F}_0$ be the density set defined as
%
\begin{equation}
\label{F0} \mathcal{F}_0 = \bigl\lbrace x\in\mathbb{R}\mapsto
\phi(x-\mu); \mu\in\mathbb{R} \bigr\rbrace.
\end{equation}
Formally, one wants to test
%
\begin{equation}
\mbox{``}f \in\mathcal{F}_0\mbox{''} \quad \mbox{against}\quad  \mbox{``}f\in\mathcal{F}_1\mbox{''}.
\label{eq:model_test}
\end{equation}

In various testing problems involving finite mixtures, the properties
of the likelihood ratio test have been widely investigated. We can
mention for instance \cite
{Azais_Gassiat_Mercadier,Chernoff_Lander,Dacunha_Gassiat,Garel} among
others. In all these papers, the main challenge is to determine the
asymptotic behaviour of the likelihood ratio under the alternative
hypothesis in order to investigate the power of the related test.
Alternative methods have also been considered: modified likelihood
ratio test \cite{Chen_Chen_Kalbfleisc}, estimation of the $L^2$
distance between the densities associated to the null and the
alternative hypotheses \cite{Charnigo_Sun}, EM approach \cite
{chen_Li} or tests based on the empirical characteristic function \cite
{klar_Meintanis}.

The main challenge related to the problem \eqref{eq:model_test} is to
find (optimal) conditions on $(\varepsilon,\mu_1,\mu_2)$ for which a
prescribed second kind error can be achieved. The first study in this
way is due to Ingster \cite{Ingster}, in the particular case where the
mean $\mu$ under the null hypothesis is known, the term $\mu_1$ in
the alternative is equal to $\mu$, and $\phi(\cdot)$ corresponds to a
Gaussian density. Similar results have also been obtained in \cite{Donoho_Jin}.
In this last paper, the so-called Higher Criticism has been
investigated. This algorithm is very powerful in the sense that it is
easy to implement, and provides similar power than the usual likelihood
ratio test. The asymptotic detection regions have been carefully
investigated in two different asymptotic regimes:
\begin{itemize}[$\bullet$]
\item[$\bullet$] the \textit{sparse regime} where
$\varepsilon\displaystyle \mathop{\sim}_{n\rightarrow+\infty} n^{-\delta}$ and
$\mu_2 - \mu_1 \displaystyle \mathop{\sim}_{n\rightarrow+\infty} \sqrt{2r
\log(n)}$ with $\frac{1}2
<\delta<1$ and $0<r<1$. In this case, it is proved that the two
hypotheses can be asymptotically separated if
\[
\lleft\{ %
\begin{array} {l@{\qquad}l} r> \delta- \frac{1} 2 & \mbox{when } \frac{1} 2 < \delta\leq\frac{3}4,
\\\noalign{\vspace*{2pt}}
r > (1 - \sqrt{1-\delta})^2 & \mbox{when } \frac{3} 4 <
\delta< 1;
\end{array} %
\rright.  %
\]
\item[$\bullet$] the \textit{dense regime} where $\varepsilon
\displaystyle \mathop{\sim}_{n\rightarrow+\infty} n^{-\delta}$ and $\mu_2 -
\mu_1 \displaystyle \mathop{\sim}_{n\rightarrow+\infty} n^{-r}$ with $0<\delta
\leq\frac{1} 2$ and $0<r<\frac{1} 2$. In this framework, the separation
is asymptotically possible if
$r< \frac{1} 2 -\delta$.
\end{itemize}
In the equations above, the notation $a_n \displaystyle \mathop{\sim}_{n\rightarrow
+\infty} b_n$ means that $\lim_{n\rightarrow+\infty} a_n/b_n =1$.
We refer for more details to \cite{Ingster} and \cite{Donoho_Jin}.
Jager and Wellner \cite{Jager_Wellner} proposed a family of tests
based on the Renyi divergences which generalizes the procedure based on
the Higher Criticism.
We also mention that generalizations of this procedure to
heteroscedastic mixtures have been proposed by Cai \textit{et al.} in \cite
{Cai_Jeng_Low} while
the problems of estimation and construction of confidence sets in
sparse mixture models are considered in \cite{Cai_Jin_Low}.
Addario-Berry \textit{et al.} \cite{Gabor} determine non-asymptotic separation
rates of testing for the contamination of a standard Gaussian vector in
$\R^n$ by non-zero mean components
when the alternatives have particular combinatorial and geometric
structures. More recently, Cai and Wu \cite{Cai_Wu} consider the
detection of sparse mixtures in the situation where the density of the
observations under the null hypothesis is fixed, but not necessarily Gaussian.

In this paper, we consider a testing problem where the null hypothesis
does not correspond to a fixed density but rather to the set of
densities $ \mathcal{F}_0 $ defined by
\eqref{F0} which corresponds to a translation model. Thus the mean
parameter $\mu$ under the null hypothesis is not assumed to be known.
The considered alternative $\mathcal{F}_1 $ corresponds to the set of
densities that are mixtures of two densities
of $ \mathcal{F}_0 $. Our aim is to decide whether the density $f$ of
the observations belongs to $\F_0$ or $\F_1$. To this end, we
introduce a new testing
procedure based on the order statistics. Contrary to the Higher
Criticism algorithm \cite{Donoho_Jin}, the main advantage of this
procedure is
that the mean $\mu$ under $H_0$ is not fixed. Since one can find
densities in $\F_1$ that are arbitrary close to $\F_0$, it is
impossible to build a
level-$\alpha$ test that achieves a prescribed power on the whole set
$\F_1$. Hence, we introduce subsets of $\F_1$ over which our
level-$\alpha$
test has a power greater than $1-\beta$. The construction of such
subsets more or less amounts to find conditions on
$(\varepsilon,\mu_1,\mu_2)$ which ensure that both hypotheses $H_0$
and $H_1$ are separable. To this end, we consider as in
\cite{Donoho_Jin} and \cite{Cai_Jeng_Low} two different regimes: the
\textit{dense} case where $|\mu_2-\mu_1|$ is assumed to be
bounded and $\varepsilon\geq C/\sqrt{n}$ for all $n\in\mathbb{N}^*$
and for some positive constant $C$, and the \textit{sparse} regime where
$\varepsilon$ is allowed to be much smaller than $1/\sqrt{n}$.

The paper is organized as follows. In Section~\ref{test:statordre}, a
testing procedure based on the order statistics 
is introduced. The Section~\ref{s:dense} is dedicated to the \textit
{dense} regime: we provide non-asymptotic lower and upper bounds for
our testing problem in the Gaussian case. Then, we investigate the
\textit{sparse} regime in Section~\ref{s:asymptotic} for both
Gaussian and Laplace distributions. Some numerical simulations,
providing a comparison with existing procedures are displayed
in Section~\ref{s:simu}. Proofs are gathered in Section~\ref
{s:proofs} and technical lemmas in the \hyperref[app]{Appendix}.

\section{The testing procedure}
\label{test:statordre}

\subsection{A test based on the order statistics}
Recall that given an i.i.d. sample $X_1,\dots, X_n$ having a common
density $f$ w.r.t. the Lebesgue measure on $\mathbb{R}$, our aim is to
consider the testing problem $H_0: f\in\F_0$ against $H_1:f\in\F
_1$, namely to decide whether $f$ corresponds to a given even density function
$\phi$ (up to a translation) or is defined as a two-components mixture
of translations of $\phi$.

In this context, one of the most popular testing procedures is the
Higher Criticism introduced in \cite{Donoho_Jin}, whose asymptotic
behaviour has been widely investigated (see also references above).
Nevertheless, there exists up to our knowledge no description of the
non-asymptotic performances of this algorithm. Moreover, this procedure
heavily depends on the knowledge of the mean under $H_0$. In this
paper, we work in a slightly different framework in the sense that a
translation model under $H_0$ is considered.

In this section, a new testing procedure based on spacings of the order
statistics is proposed. The order statistics are denoted by
$X_{(1)}\leq X_{(2)}\leq\cdots\leq X_{(n)}$. The main underlying idea
is that the spacing of these order statistics are free with respect to
the mean under $H_0$: for some $k<l\in\lbrace1,\dots, n \rbrace$,
the mean value affects the spatial position of a given $X_{(k)}$, but
not $X_{(l)} - X_{(k)}$.
Moreover, the distribution of the variables $X_{(l)} - X_{(k)}$ is
known under $H_0$ and has a different behavior under $H_1$, provided
$k$ and $l$ are well-chosen.

Let $\alpha\in\,]\,0,1[$ be a fixed level, $\P_f$ the distribution of
$X_1,\dots, X_n$ having common density $f$, and
$ \E_f$ the corresponding expectation.
In the following, a level-$\alpha$ test function $T_\alpha$ denotes a
measurable function of $(X_1,\ldots,X_n)$ with values in $\{0,1\}$,
such that the null hypothesis is rejected if $T_\alpha=1$ and
$\sup_{f\in\F_0} \P_f(T_\alpha=1)\leq\alpha$. Assume
that $n\geq2$ and consider the subset $\K_n$ of $\{1, 2, \ldots, n/2
\}$ defined by
\[
\K_n =\bigl\{2^j, 0\leq j\leq \bigl[{
\log_2(n/2)} \bigr] \bigr\}.
\]
Our test statistics is defined as
%
\begin{equation}
\label{def:testStatOrd} \Psi_{\alpha} := \sup_{k \in\K_n} \{{\mathds
{1}_{X_{(n-k+1)} - X_{(k)} > q_{\alpha_n,k}}} \},
\end{equation}
where, for all $u \in\,]\,0,1[$, $q_{u,k}$ is the $(1-u)$-quantile of
$X_{(n-k+1)} - X_{(k)}$ under the null hypothesis and
\[
\alpha_n=\sup \bigl\{{u \in\,]\,0,1[, \P_{H_0} ({\exists k \in
\K _n, X_{(n-k+1)}-X_{(k)} > q_{u,k}} ) \leq
\alpha} \bigr\}.
\]
Note that, by construction, $\alpha_n\leq\alpha$. Since the
distribution of
$X_{(n-k+1)} - X_{(k)}$ under the null hypothesis is independent of the
mean value $\mu$ of the $X_i$'s, $q_{\alpha_n,k}$ and $\alpha_n$ can
be approximated (via Monte-Carlo simulations for instance) under the
assumption that the $X_i$'s have common density $\phi$. Below (see in
particular Section~\ref{s:general_result}), we also provide explicit
upper bounds for the quantiles, which can be used instead of the true
$q_{\alpha,k}$ if necessary.

\subsection{First and second kind errors}

By definition, the test statistics $\Psi_{\alpha}$ introduced in
\eqref{def:testStatOrd} is exactly of level $\alpha$, namely
\[
\P_{H_0}(\Psi_\alpha= 1) = \P_{H_0} ({\exists k \in
\K_n, X_{(n-k+1)}-X_{(k)} > q_{\alpha_n,k}} ) \leq
\alpha,
\]
thanks to the definition of $\alpha_n$. We point out that $\alpha_n
\geq\alpha/ |\mathcal{K}_n |$, where $|\mathcal{K}_n |$ denotes the
cardinality of $\mathcal{K}_n $. Indeed,
\begin{eqnarray*}
\P_{H_0} ({\exists k \in\K_n, X_{(n-k+1)}-X_{(k)}
> q_{\alpha/
|\mathcal{K}_n |,k}} ) & \leq& \sum_{k\in\K_n}
\P_{H_0}(X_{(n-k+1)}-X_{(k)} > q_{\alpha/
|\mathcal{K}_n |,k})
\\
& \leq& \sum_{k\in\K_n} \frac{\alpha}{|\K_n|} \leq\alpha.
\end{eqnarray*}
In practice, the choice of $\alpha_n$, instead of the so-called
Bonferroni correction $ \alpha/ |\mathcal{K}_n |$, allows a numerical
improvement of the performances of $\Psi_\alpha$. We refer to \cite
{FL_2006} for an extended discussion on this subject.

Now, we turn our attention to the control of the second kind error. We
emphasize that the test $\Psi_\alpha$ is a multiple testing
procedure: we combine $|\K_n|$ different tests, which correspond to
different spacing
for the order statistics. We can remark that, for any $f\in\mathcal{F}_1$
\begin{eqnarray*}
\mathbb{P}_{f}(\Psi_\alpha=0) & = & \mathbb{P}_{f}
\Bigl( \sup_{k \in\mathcal{K}_n} \lbrace \mathbh{1}_{X_{(n-k+1)} - X_{(k)} > q_{\alpha_n,k}} \rbrace=0
\Bigr)
\\
& = & \mathbb{P}_{f} \biggl( \bigcap_{k \in\mathcal{K}_n}
\lbrace \mathbh{1}_{X_{(n-k+1)} - X_{(k)} > q_{\alpha_n,k}} =0 \rbrace \biggr)
\\
& \leq& \inf_{k \in\mathcal{K}_n} \mathbb{P}_{f} ( \mathbh
{1}_{X_{(n-k+1)} - X_{(k)} > q_{\alpha_n,k}} =0 ).
\end{eqnarray*}
Hence, the second kind error of $\Psi_\alpha$ is close to the
smallest one in the collection $\K_n$. In some sense, the ``optimal''
choice of $k\in\K_n$
is data-driven. The only price to pay for adaptation relies in the
``level'' $\alpha_n$, which is smaller than $\alpha$.

From now on, our aim is to evaluate precisely the power of the test for
different kinds of alternatives:
dense mixtures (Section~\ref{s:dense}) or sparse mixtures
(Section~\ref{s:asymptotic}). A general non-asymptotic result is
provided in
Section~\ref{s:general_result}.

\section{Dense mixtures}
\label{s:dense}

In this section, we assume that the difference between the means $\mu
_1$ and $\mu_2$ of the two components
of the mixture is bounded. We will see that the settings of interest
correspond to the case where $\varepsilon\geq C/\sqrt{n}$ for some
constant $C>0$.
In the literature, this regime is called the dense case.

We consider the set of alternatives
\[
\F_1[M] = \bigl\{{f(\cdot)=(1-\varepsilon)\phi(\cdot-\mu_1) +
\varepsilon \phi(\cdot-\mu_2), \varepsilon\in\,]\,0,1[, 0 < \mu_2-
\mu_1 \leq M } \bigr\} %
\]
with $M>0$.
When the density of the standard normal distribution is considered
($\phi=\phi_G$), this set is denoted $\F_{1,G}[M]$.

The aim of this section is to provide explicit conditions on the
triplet $(\varepsilon,\mu_1,\mu_2)$ that guarantee
a prescribed power for a test of mixture detection, provided that $f
\in \F_1[M]$.
More precisely, we measure the distance to the null hypothesis by the
quantity $ d(\varepsilon, \mu_1,\mu_2)= \varepsilon(1-\varepsilon)
(\mu_2-\mu_1)^2 $ and we assume that
$d(\varepsilon, \mu_1,\mu_2) \geq\rho$ for some $\rho>0$. The
question can be
therefore formulated as follows: what is the minimal value of $\rho$
to be able to detect the mixture? Under this condition, is
the test proposed in Section~\ref{test:statordre} powerful? We address
these two questions for Gaussian mixture
models. We also provide a simple test based on the estimation of the
variance which is powerful (not only for Gaussian mixtures) in the
framework considered in this section.

\subsection{Lower bound for the detection of a Gaussian mixture model}
In this section, we consider the same definition of non-asymptotic
lower bounds for hypotheses testing problems
than the ones introduced in \cite{Yannick}
for signal detection in a Gaussian regression model or a Gaussian
sequence model. Let us recall these definitions.
Given $\beta\in\,]\,0,1[$, the class
of alternatives $\F_1[M]$, and a level-$\alpha$ test $T_{\alpha}$
with values in $\{0,1\}$ (rejecting $H_0$
when $T_{\alpha}=1$), we define the uniform separation rate $\rho
(T_{\alpha},\F_1[M],\beta)$ of $T_{\alpha}$ over
the class $\F_1[M] $ as the smallest positive number $\rho$ such that
the test has a second kind error
at most equal to $\beta$ for all alternatives $f$ in $\F_1[M]$ such
that $ d(\varepsilon, \mu_1,\mu_2)=
\varepsilon(1-\varepsilon) (\mu_2-\mu_1)^2 \geq\rho$.
More
precisely,
%
\begin{equation}
\label{defvitesse} \rho\bigl(T_\alpha,\F_1[M],\beta\bigr)=\inf{}
\Bigl\{\rho>0, \sup_{f\in\F
_1[M], d(\varepsilon, \mu_1,\mu_2) \geq
\rho}\P_f(T_{\alpha}=0)
\leq\beta \Bigr\}.
\end{equation}
Then, we introduce the $(\alpha,\beta)$-minimax separation rate over
$\F_1[M]$ defined as
%
\begin{eqnarray}
\label{mrt} \underline{\rho}\bigl(\F_1[M] ,\alpha,\beta\bigr)=
\inf_{T_{\alpha}}\rho \bigl(T_{\alpha},\F_1[M],\beta
\bigr),
\end{eqnarray}
where the infimum is taken over all level-$\alpha$ tests $ T_{\alpha}$.

We provide in the next theorem a non-asymptotic lower bound for
$\underline{\rho}(\F_1[M] ,\alpha,\beta)$ in the case where $\phi$
corresponds to the standard Gaussian density.

%
\begin{them}
\label{Th:lowerbound:dense}
Let $\alpha\in\,]\,0,1[$ and $\beta\in\,]\,0,1-\alpha[$. Let
\[
\rho^\star=\frac{1}{C(M)} \biggl({\sqrt{\frac{-2\log[c(\alpha
,\beta)]}{n}} \sqrt{1 +
\frac{\log[c(\alpha,\beta)]}{2n}}} \biggr),
\]
with $c(\alpha,\beta)=1-\frac{(1-\alpha-\beta)^2}{2}$ and $C(M) =
\sqrt{\frac{1}{2} + \frac{2M^2}{3} \mathrm{e}^{M^2/4}}$. Then for all
$\rho< \rho^\star$,
\[
\inf_{T_\alpha} \sup_{f\in\F_{1,G}[M], d(
\varepsilon, \mu_1,\mu_2) \geq\rho}
\P_f(T_\alpha=0) > \beta,
\]
where the infimum is taken over all level-$\alpha$ test $T_\alpha$.
This implies that
\[
\underline{\rho}\bigl(\F_{1,G}[M] ,\alpha,\beta\bigr) \geq
\rho^\star.
\]
\end{them}

Theorem~\ref{Th:lowerbound:dense} implies that whatever the
level-$\alpha$ test $T_\alpha$, if
$\rho<\rho^\star$, there exists a density $f\in\F_{1,G}[M]$ for
which $\P_f(T_\alpha=0) > \beta$.
In particular, testing is not possible if $\mu_2 -\mu_1$ is too small
with respect to $\varepsilon(1-\varepsilon)$. We will show in Section~\ref
{SecVar} that this condition on $(\varepsilon,\mu_1,\mu_2)$ is optimal
(up to constant).

\subsection{Upper bound for the testing procedure \texorpdfstring{$\Psi_\alpha$}{$Psi_alpha$} in
the Gaussian case}\label{s:main}
The goal of this section is to give explicit conditions on
$(\varepsilon,\mu_1,\mu_2)$ that ensure a prescribed power for
the test $\Psi_{\alpha}$ defined in (\ref{def:testStatOrd}), when
$\phi$ is the standard Gaussian density.

%
\begin{them}
\label{Th:majo}
Let $X_1,\ldots,X_n$ be i.i.d. real random variables with common
density $f$. Let $\alpha\in\,]\,0,1[$ and consider the level-$\alpha$
test $\Psi_{\alpha}$
defined by \eqref{def:testStatOrd}. Let $\beta\in\,]\,0,1-\alpha[$ and
$M>0$. Assume that $n$ fulfills $ n \geq3$ and
$8.25\times{\log(4\log_2(n/2)/\alpha)}/{n} \leq\int_M^{\infty}
\phi_G(x) \,\mathrm{d}x $.

Then, there exists a positive constant $ C(\alpha,\beta,M) $
depending only on $\alpha$, $\beta$ and M, such that if
%
\begin{equation}
\label{Conddense} \rho\geq C(\alpha,\beta,M) \sqrt{\frac{\log\log(n)}{n}},
\end{equation}
then,
\[
\sup_{f\in\F_{1,G}[M], d(\varepsilon, \mu_1,\mu_2) \geq\rho} \P _f(\Psi_{\alpha}=0)\leq\beta.
\]
\end{them}

\begin{Comments*}
The technical condition on $n$ to get the result of Theorem~\ref
{Th:majo} is satisfied for $n \geq107$ when $M=1/10$ and $\alpha=0.05$.

Note that the value of $\rho$ proposed in \eqref{Conddense} differs
from the lower bound $\rho^\star$ by a term of order
$\sqrt{\log\log n}$. This log log term is due to the multiple
(adaptive) testing procedure: the optimal value for $k\in\mathcal{K}_n$
in the test $\Psi_\alpha$ is chosen from the data. Hence, this $
\sqrt{\log\log(n)} $ term corresponds to the price to pay in such a setting.
This kind of logarithmic loss is quite classical in test theory: see
for instance \cite{Spok} or \cite{FL_2006} in slightly different
settings.

Instead of considering the test statistics $\Psi_{\alpha}$ defined by
\eqref{def:testStatOrd}, we could introduce the statistics
\[
\mathds{1}_{X_{(n-k^*+1)} - X_{(k^*)} > q_{\alpha,k^*}},
\]
where $k^*$ has to be suitably chosen and depends on $M$. By this way,
we would avoid the logarithmic loss in the minimax separation rate
over the set $ \F_{1,G}[M]$ and obtain a rate that coincides (up to
constants) with the lower bound given in Theorem~\ref{Th:lowerbound:dense} (see the proof of Theorem~\ref{Th:majo}).
In practice, using the test statistics $\Psi_{\alpha}$ is more
satisfactory since it does
not depend on $M$.
\end{Comments*}

\subsection{A testing procedure based on the variance}\label{SecVar}
In this section, we do not assume that the $X_i$'s are Gaussian random
variables.
We are interested in a simple test based on the variance of the
$X_i$'s. We will prove that this test allows us
to achieve the lower bound obtained in Theorem~\ref
{Th:lowerbound:dense}. 

Remark that under $H_0$, $\Var(X_i) = \sigma^2$, where
$\sigma^2= \int_{\R} x^2 \phi(x) \,\mathrm{d}x$, while under $H_1$, $\Var(X_i) = \sigma^2
+ \varepsilon(1-\varepsilon) (\mu_2-\mu_1)^2$. Hence, we consider
the test $\psi_\alpha$ defined by
%
\begin{eqnarray}\label{eq:test_variance}
\psi_\alpha= \mathbh{1}_{\lbrace S_n^2 > v_{\alpha,n} \rbrace},\qquad  \mbox{where }
S_n^2 = \frac{1}{n-1} \sum
_{i=1}^n (X_i - \bar
X_n)^2,
\end{eqnarray}
and $v_{\alpha,n}$ denotes the $(1-\alpha)$-quantile of the variable
$S_n^2$ under $H_0$. Then the following proposition holds.

%
\begin{prop}
\label{Prop:upperbound:dense}
Let $\alpha\in\,]\,0,1[$ and $\beta\in\,]\,0,1-\alpha[$. Assume that the
density function $\phi$ has a finite fourth moment: $\int_{\R} x^4
\phi(x) \,\mathrm{d}x
\leq B$. There exists a positive constant $C(\alpha,\beta,M,B) $
depending on $( \alpha,\beta,M,B)$ such that if
%
\begin{equation}\label{rho:testvariance}
\rho\geq C(\alpha,\beta,M,B)/\sqrt{n},
\end{equation}
then
\[
\sup_{f\in\F_{1}[M], d(\varepsilon, \mu_1,\mu_2) \geq\rho} \P _f(\psi_{\alpha}=0)\leq\beta.
\]
\end{prop}

In the Gaussian case, $ \int_{\R} x^4 \phi_G(x) \,\mathrm{d}x =3$. Hence,
Proposition~\ref{Prop:upperbound:dense} assesses the optimality of the lower
bound given in Theorem~\ref{Th:lowerbound:dense}. Note that the value
of $\rho$ proposed in \eqref{rho:testvariance} differs from
$\rho^\star$ by constant. Finding the optimal constant for our
testing problem is a very difficult question that is out of the scope of
this paper. For interested reader, we mention the work of \cite
{Ingster} in a slightly different (asymptotic) setting.

The result given in Proposition~\ref{Prop:upperbound:dense} seems more
efficient than the one stated in Theorem~\ref{Th:majo} since
the condition to control by $\beta$ the second kind error is
$\varepsilon(1-\varepsilon) (\mu_2-\mu_1)^2 >C/\sqrt{n}$ instead
of $C\sqrt{\log\log(n)}/\sqrt{n}$. Nevertheless, the test based on
the variance would fail in the asymptotic sparse regime (see
Sections~\ref{s:asymptotic} and \ref{SecVarDiscussion} for more
details). This is not satisfactory from a practical point
of view since our aim is to provide a testing procedure which adapts to
all possible situations.

\subsection{An asymptotic study}

The results stated in Theorems \ref{Th:lowerbound:dense} and \ref
{Th:majo} are non-asymptotic. In this section, we will adopt an
asymptotic point of view for our testing problem in the Gaussian
setting. As in \cite{Donoho_Jin}, we will work with the following
parametrization
%
\begin{equation}\label{equ:dense_regime}
\varepsilon\mathop{\sim}_{n\rightarrow+\infty}  n^{-\delta}\quad  \mbox{and}\quad
\mu_2 - \mu_1 \mathop{\sim}_{n\rightarrow+\infty}
n^{-r} \qquad \mbox{with } 0<\delta\leq\tfrac{1}{2} \mbox{ and } 0<r<
\tfrac{1}{2}.
\end{equation}

%
\begin{cor} \label{Corollaire:dense}
The detection boundary in the \textit{dense} regime \eqref
{equ:dense_regime} is $r^*(\delta) = \frac{1} 4 - \frac\delta2$: the
detection is possible when $r<r^*(\delta) = \frac{1} 4 - \frac\delta
2$ and impossible if $r>r^*(\delta)$.

In particular, setting $ f(\cdot)=(1-\varepsilon)\phi_G(\cdot-\mu_1) +
\varepsilon\phi_G(\cdot-\mu_2) $, we have, for $n$ large enough,
\[
\P_f(\Psi_{\alpha}=0)\leq\beta\quad  \mbox{and}\quad
\P_f(\psi_{\alpha
}=0)\leq\beta, %
\]
provided $r<r^*(\delta)$, where the tests $\Psi_\alpha$ and $\psi
_\alpha$ are respectively, defined in \eqref{def:testStatOrd} and
\eqref{eq:test_variance}
\end{cor}

The proof of Corollary~\ref{Corollaire:dense} is omitted since it can
be obviously deduced from Theorems \ref{Th:lowerbound:dense} and \ref
{Th:majo}. These results are therefore different from the one obtained
in a dense regime in a contamination framework where one wants to test
$H_0:f=\phi_G(\cdot)$ against $H_1: f\in\{(1-\varepsilon)\phi_G(\cdot) +
\varepsilon\phi_G(\cdot-\mu); \varepsilon\in\,]\,0,1[, \mu\in\mathbb
{R}\}$. In this case, as mentioned in introduction, the detection is
possible in the dense regime for $r<\frac{1} 2 - \delta$
(see \cite{Ingster,Donoho_Jin}). This difference is due to the fact that
the mean under $H_0$ is unknown, which makes the testing problem harder.

\section{Sparse mixtures}
\label{s:asymptotic}

In the previous part, we have considered the case where the term $\mu
_2 - \mu_1$ is bounded under the alternative hypothesis. In this
section, we will consider the situation where this quantity is allowed
to tend to infinity as $n$ increases. It appears that in such a
framework, the most interesting cases correspond to the situation where
$\varepsilon\ll \frac{1}{\sqrt{n}}$ as $n\rightarrow+\infty$. In
the literature, this regime is called the sparse case.

This setting has been considered for several different kinds of
distributions. In particular, optimal separation conditions on the
behavior of $\mu_2-\mu_1$ as $n\rightarrow+\infty$ have been
displayed in various situations. In the following, we prove that our
testing procedure provides a satisfying behavior in this sparse
setting: in particular, we prove that it reaches the optimal separation
conditions established in \cite{Donoho_Jin} in both the Gaussian and
the Laplace cases.

\subsection{The Gaussian case}
Let $\mathcal{F}_0$ and $\mathcal{F}_1$ be the sets defined by \eqref
{F0} and \eqref{F1} respectively. Given an i.i.d. sample $X_1,\dots,
X_n$ having a common density $f$, we test in this part
\[
\mbox{``}f\in\mathcal{F}_0\mbox{''} \quad \mbox{against}\quad  \mbox{``}f\in\mathcal{F}_1\mbox{''},
\]
in the particular case where $\phi(\cdot)=\phi_G(\cdot)$, the standard
Gaussian density. In this setting, the so-called \textit{sparse}
regime introduced in \cite{Donoho_Jin} is characterized by
%
\begin{equation}\label{equ:sparse_regime}
\varepsilon\mathop{\sim}_{n\rightarrow+\infty}  n^{-\delta} \quad \mbox{and}\quad
\mu_2 - \mu_1 \mathop{\sim}_{n\rightarrow+\infty}
\sqrt{2r \log(n)} \qquad \mbox{with } \tfrac{1} 2 <\delta<1 \mbox{ and } 0<r<1.
\end{equation}
Below, we analyze the performances of our testing procedure \eqref
{def:testStatOrd} in this \textit{sparse} regime. The corresponding
proof is provided in Section~\ref{s:preuve_sparse_gaussien}.

%
\begin{them}\label{Prop:sparse:statord}
Let $X_1,\ldots,X_n$ be i.i.d. real random variables with common
density $f$. Let $\alpha\in\,]\,0,1[$ and consider the level-$\alpha$
test $\Psi_{\alpha}$
defined by \eqref{def:testStatOrd}. We consider the case where $\phi=
\phi_G$.

We assume that the behavior of $(\varepsilon,\mu_1,\mu_2)$ is governed
by \eqref{equ:sparse_regime} and that $r>r^*(\delta) $ with
\[
r^*(\delta) = \lleft\{ %
\begin{array} {l@{\qquad} l} \delta- \frac{1} 2
& \mbox{ if } \frac{1} 2 < \delta<\frac{3} 4,
\\\noalign{\vspace*{2pt}}
(1-\sqrt{1-\delta})^2 & \mbox{ if } \frac{3} 4 \leq\delta<1.
\end{array} %
\rright.
\]
Then, setting
$ f(\cdot)=(1-\varepsilon)\phi_G(\cdot-\mu_1) + \varepsilon\phi_G(\cdot-\mu
_2) $, we have, for $n$ large enough,
\[
\P_f(\Psi_{\alpha}=0)\leq\beta. %
\]
\end{them}

In the sparse regime, we exactly recover the separation boundaries that
are already known in the case where the null hypothesis is reduced to a
standard normal density, and the alternative is the mixture $
(1-\varepsilon)\phi_G(\cdot) + \varepsilon\phi_G(\cdot-\mu) $. Hence, the
fact that the mean under $H_0$ is unknown does not affect the
difficulty of the related testing problem in this specific framework.

This proves the optimality of our procedure in the sparse regime.
Indeed, the lower bounds established by \cite{Ingster,Cai_Jeng_Low} in
the case where the null hypothesis is reduced to the standard Gaussian
density also provide lower bounds for our testing problem.
This comes from the fact that
\begin{itemize}[$\bullet$]
\item[$\bullet$] a level-$\alpha$ test for our testing problem is also a
level-$\alpha$ test for testing the null hypothesis ``$f=\phi_G$'',
\item[$\bullet$] the case where the null hypothesis is reduced to the centered
Gaussian density is included in our setting.
\end{itemize}

\subsection{The Laplace case}
In this section, we address the testing problem \eqref{eq:model_test}
in the particular case where $\phi$ corresponds to the Laplace
density, namely $\phi=\phi_L$ where
\[
\phi_L(x) = \tfrac{1}{2} \mathrm{e}^{-|x|},\qquad  \forall x\in
\mathbb{R}.
\]
In other words, given a sample $X_1,\dots,X_n$, our aim is to test
whether the underlying density is $\phi_L(\cdot-\mu)$ for some unknown parameter
$\mu$ or $(1-\varepsilon)\phi_L(\cdot-\mu_1) + \varepsilon\phi
_L(\cdot-\mu_2)$ in the particular case where $\varepsilon= \mathrm{o}(1/\sqrt
{n})$ as $n\rightarrow+\infty$.

In this context, \cite{Donoho_Jin} have proved that the cases of
interest in the \textit{sparse} regime correspond to the following
parametrization
%
\begin{equation}\label{equ:sparse_Laplace}
\varepsilon\mathop{\sim}_{n\rightarrow+\infty}  n^{-\delta}\quad  \mbox{and}\quad
\mu_2 - \mu_1 \mathop{\sim}_{n\rightarrow+\infty}  r
\log(n) \qquad \mbox{with } \tfrac{1} 2 <\delta<1 \mbox{ and } 0<r<1.
\end{equation}
The performances of our testing procedure \eqref{def:testStatOrd} are
described in the following theorem, whose proof is given in
Section~\ref{s:preuve_sparse_laplace}.

%
\begin{them}\label{Prop:sparse:Laplace}
Let $X_1,\ldots,X_n$ be i.i.d. real random variables with common
density $f$. Let $\alpha\in\,]\,0,1[$ and consider the level-$\alpha$
test $\Psi_{\alpha}$
defined by \eqref{def:testStatOrd}. We consider the case where $\phi=
\phi_L$.

We assume that the behavior of $(\varepsilon,\mu_1,\mu_2)$ is governed
by \eqref{equ:sparse_Laplace} and that $r>r^*(\delta) $ with
\[
r^*(\delta) = 2\delta- 1.
\]
Then, setting
$ f(\cdot)=(1-\varepsilon)\phi_L(\cdot-\mu_1) + \varepsilon\phi_L(\cdot-\mu
_2) $, we have, for $n$ large enough,
\[
\P_f(\Psi_{\alpha}=0)\leq\beta. %
\]
\end{them}

Remark that the detection boundary $r^*(\delta)$ is the same that have
been exhibited by \cite{Donoho_Jin}. Once again, these lower bounds
remain valid since:
\begin{itemize}[$\bullet$]
\item[$\bullet$] a level-$\alpha$ test for our testing problem is also a
level-$\alpha$ test for testing the null hypothesis ``$f=\phi_L$'',
\item[$\bullet$] the case where the null hypothesis is reduced to the centered
Laplace density is included in our setting.
\end{itemize}

\subsection{The variance test for sparse mixtures: A heuristic discussion}
\label{SecVarDiscussion}

We point out that the testing procedure introduced in Section~\ref
{SecVar} will not be convenient in this asymptotic sparse setting.
Indeed, we can remark that
\[
\Var_{\phi}(X_i) = \int_\mathbb{R}
x^2 \phi(x) \,\mathrm{d}x,
\]
while, for any $f=(1-\varepsilon)\phi(\cdot-\mu_1) + \varepsilon\phi
(\cdot-\mu_2) $
\[
\Var_{f}(X_i) = \int_\mathbb{R}
x^2 \phi(x)\,\mathrm{d}x + \varepsilon (1-\varepsilon) (\mu_1 -
\mu_2)^2.
\]
For both Gaussian and Laplace mixtures, in the respective asymptotic
schemes \eqref{equ:sparse_regime} and \eqref{equ:sparse_Laplace}, we
get that
\[
\Var_{f}(X_i) - \Var_{\phi}(X_i)
= \varepsilon (1-\varepsilon) (\mu_1-\mu_2)^2
\ll \frac{1}{\sqrt{n}}, \qquad \mbox{as } n\rightarrow+\infty.
\]
Since the variance is estimated at a parametric ``rate'' $1/\sqrt{n}$,
the test $\psi_\alpha$ introduced in \eqref{eq:test_variance} will fail
in this setting: it will not be able to separate $H_0$ from $H_1$ with
an appropriate power.

\section{Simulation study}
\label{s:simu}

In this section, we provide some numerical experiments in order to
enhance the performances of our testing procedure $\Psi_\alpha$.
Comparisons with the Higher Criticism and the Kolmogorov--Smirnov test
are provided. Since these both procedures are not designed for the
considered framework (translated model with unknown mean),
straightforward modifications are proposed. We have also included in
these numerical experiments the test based on the variance defined in
Section~\ref{SecVar}.

\subsection{Contamination of \texorpdfstring{$\phi_G$}{$phi_G$}}
In this section, we deal with the framework considered in \cite
{Donoho_Jin}: the mean under $H_0$ is assumed to be known (equal to 0)
and equal to $\mu_1$. More formally, given $(X_1,\ldots,X_n)$, i.i.d.
random variables with an unknown density function $f$, our aim is to test
%
\begin{equation}\label{eq:model_donoho}
H_0: f(\cdot)=\phi_G(\cdot) \mbox{ against } H_1:
f\in\bigl\{x\mapsto (1-\varepsilon)\phi_G(x) + \varepsilon
\phi_G(x-\mu); \mu\in \mathbb{R}, \varepsilon\in\,]\,0,1[\bigr\}.
\end{equation}
In this case, our testing procedure $\Psi_\alpha$ described in \eqref
{def:testStatOrd} can be easily adapted as follows:
\[
\tilde\Psi_{\alpha} = \sup_{k \in\K_n} \{ \mathds
{1}_{X_{(n-k+1)} > q_{\alpha_n,k}} \}, %
\]
where $q_{\alpha,k}$ is the $(1-\alpha)$-quantile of $X_{(n-k+1)}$
under the null hypothesis, $\K_n=\{2^j; 0\leq j \leq[\log_2(n/2)]\}$ and
\[
\alpha_n=\sup\bigl\{u \in\,]\,0,1[, \mathbb{P}_{H_0} (\exists k
\in\K _n, X_{(n-k+1)} > q_{u,k} ) \leq\alpha\bigr\}.
\]
For the sake of brevity, we do not exhibit a theoretical study of the
performances of this procedure for the testing problem \eqref
{eq:model_donoho}. Indeed, the methodology is rather close to the one
proposed in this paper, up to some technical modifications. It is
possible to see that this procedure achieves the optimal asymptotic
separation set in both the \textit{dense} and \textit{sparse}
regimes, as described in \cite{Donoho_Jin}.

The power of our testing procedure is compared with the one of
\begin{itemize}[$\bullet$]
\item[$\bullet$] Kolmogorov--Smirnov test:
The level-$\alpha$ test function is $\psi_{\mathrm{KS},\alpha} = \mathds
{1}_{T_{\mathrm{KS}}> q_{\mathrm{KS},\alpha}}$ where
\[
T_{\mathrm{KS}}=\sup_{x\in\mathbb{R}}  \sqrt{n}\bigl|F_n(x) -
\Phi_G(x)\bigr| %
\]
with the empirical distribution function $F_n(x)=\frac{1} n
\sum^{n}_{i=1} \mathds{1}_{X_i\leq x}$,
$\Phi_G$ the cumulative distribution function of the standard Gaussian
variable, and $q_{\mathrm{KS},\alpha}$ is the $(1-\alpha)$ quantile of
$T_{\mathrm{KS}}$ under $H_0$.

\item Higher Criticism \cite{Donoho_Jin}:
Let $p_i=\mathbb{P}(Z>X_i)$ where $Z\sim\mathcal{N}(0,1)$ for all
$i\in\{1,\ldots,n\}$ and $p_{(1)}\leq p_{(2)}\leq\cdots\leq p_{(n)}$.
This test is based on
\[
\operatorname{HC} =  \max_{1\leq i \leq n} \frac{\sqrt{n}  (\sfrac{i}{n} - p_{(i)} )}{\sqrt{p_{(i)} (1-p_{(i)})}}. %
\]
The level-$\alpha$ test function is $\psi_{\mathrm{HC},\alpha} = \mathds
{1}_{\mathrm{HC}> q_{\mathrm{HC},\alpha}}$ where $q_{\mathrm{HC},\alpha}$ is the $(1-\alpha)$
quantile of $\operatorname{HC}$ under $H_0$.
\item The test based on the variance (see Section~\ref{SecVar}).\vadjust{\goodbreak}
\end{itemize}

In order to study the power of these testing procedures, a Monte-Carlo
procedure is considered with $N=100\,000$ samples of size $n=100$ from a
mixture distribution $(1-\varepsilon)\phi_G(\cdot) +\linebreak[4]  \varepsilon\phi
_G(\cdot-\mu)$ with $\varepsilon\in\{0.05,0.15,0.25,0.35,0.45\}$ and
$\mu\in[0,10]$. The power functions of these testing procedures in
the different scenarios are reported in Figure~\ref{Fig:Power:contam}.

%
\begin{figure}[b]

\includegraphics{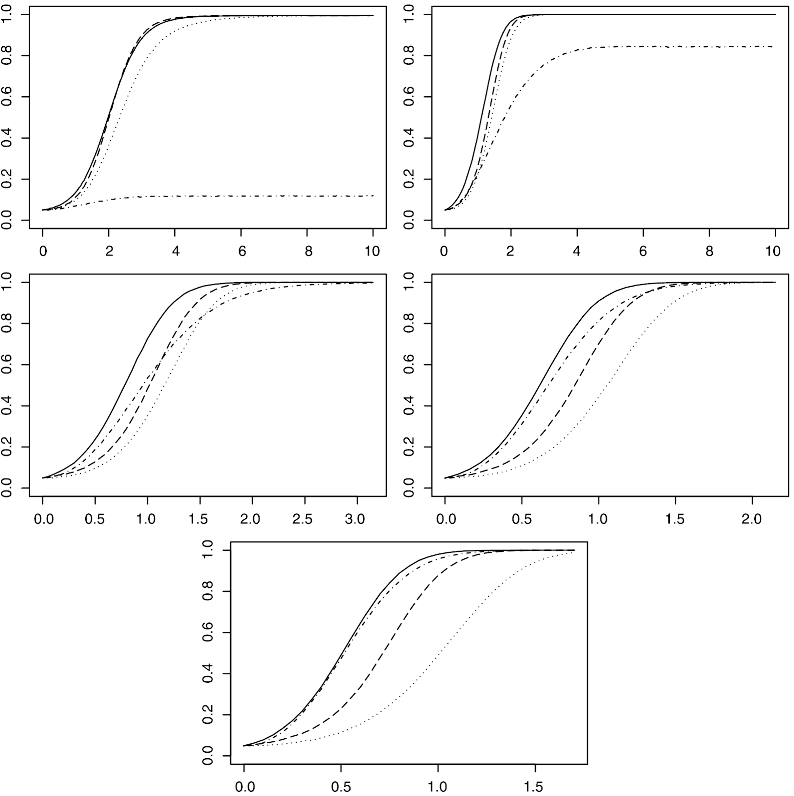}

\caption{Power function of the three considered testing procedures
(continuous line for our test $\tilde\Psi_\alpha$, dashed line for
Higher Criticism, dashed/dotted line for the Kolmogorov--Smirnov test
and dotted line for the test based on the variance) according to $\mu
$, for $\varepsilon= 0.05$ (top left), $0.15$ (top right), $0.25$
(middle left), $0.35$ (middle right) and $0.45$ (bottom) in a
contamination framework.}
\label{Fig:Power:contam}
\end{figure}

It appears that our procedure performs as well as the Higher Criticism
when $\varepsilon$ is small w.r.t. the size of the sample, while the
Kolmogorov--Smirnov test possesses a bad behavior. Such a setting is
close to the
\textit{sparse} regime. Nevertheless, the performances of the Higher
Criticism deteriorates as $\varepsilon$ increases while the power of our
test $\tilde\Psi_\alpha$ remains stable. In this setting, the test
based on the variance does not perform very well. The main reason is
that, in this case, the mean under $H_0$ is known. Hence, a test based
on the empirical mean of the observations would be more appropriate.

\subsection{Gaussian mixtures with unknown means}\label{subsection:simu2}
In this section, we deal with our testing problem. A simulation study
is proposed in order to investigate the power of our testing procedure
$\Psi_\alpha$ described by \eqref{def:testStatOrd}.
Our testing
procedure is compared with the following adaptations of
Kolmogorov--Smirnov test and Higher Criticism:
\begin{itemize}[$\bullet$]
\item[$\bullet$] Kolmogorov--Smirnov test:
The level-$\alpha$ test function is $\widehat{\psi}_{\mathrm{KS},\alpha} =
\mathds{1}_{\hat{T}_{\mathrm{KS}}> \hat{q}_{\mathrm{KS},\alpha}}$ where
\[
\hat{T}_{\mathrm{KS}}=\sup_{x\in\mathbb{R}}  \sqrt{n}\bigl|F_n(x)
- \Phi_G(x-\bar X)\bigr| %
\]
with the empirical mean $\bar X$, the empirical distribution function
$F_n(x)=\frac{1} n \sum_{i=1}^{n} \mathds
{1}_{X_i\leq x}$, and $\hat{q}_{\mathrm{KS},\alpha}$ is the $(1-\alpha)$
quantile of $\hat{T}_{\mathrm{KS}}$ under $H_0$.

\item[$\bullet$] Higher Criticism \cite{Donoho_Jin}:
Let $\hat{p}_i=\mathbb{P}(Z-\bar X>X_i)$ where $Z\sim\mathcal
{N}(0,1)$ for all $i\in\{1,\ldots,n\}$ and $\hat{p}_{(1)}\leq\hat
{p}_{(2)}\leq\cdots\leq\hat{p}_{(n)}$.
This test is based on
\[
\widehat{\operatorname{HC}} = \max_{1\leq i \leq n}  \frac{\sqrt{n}
 (\sfrac{i}{n} - \hat{p}_{(i)} )}{\sqrt{\hat{p}_{(i)}
(1-\hat{p}_{(i)})}}. %
\]
The level-$\alpha$ test function
 is $\hat{\psi}_{\mathrm{HC},\alpha} =
\mathds{1}_{\widehat{\mathrm{HC}}> \hat{q}_{\mathrm{HC},\alpha}}$ where $\hat
{q}_{\mathrm{HC},\alpha}$ is the $(1-\alpha)$ quantile of $\widehat{\operatorname{HC}}$
under $H_0$.
\item[$\bullet$] The test based on the variance (see Section~\ref{SecVar}).
\end{itemize}

In order to study the power of these testing procedures, a Monte-Carlo
procedure is considered with $N=100\,000$ samples of size $n=100$ from a
mixture distribution $(1-\varepsilon)\phi_G(\cdot) + \varepsilon\phi
_G(\cdot-\mu_2)$ with $\varepsilon\in\{0.05,0.15,0.25,0.35,0.45\}$. We
deal with $\mu_1=\mu=0$ and $\mu_2\in[0,10]$. The power functions
of these testing procedures in the different scenarios are reported in
Figure~\ref{Fig:Power:unknownmean}.
%
\begin{figure}

\includegraphics{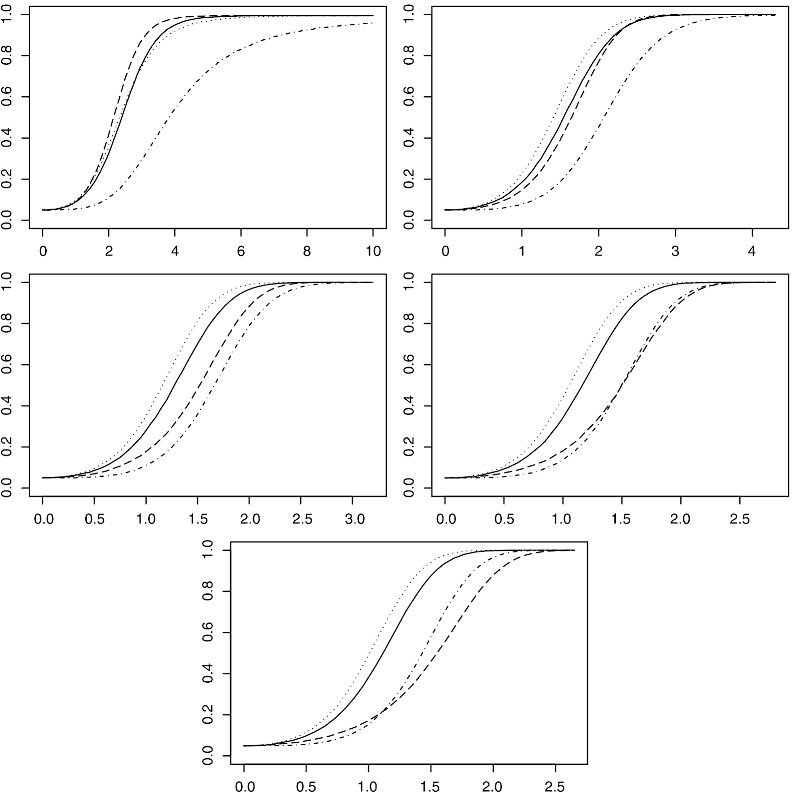}

\caption{Power function of the three considered testing procedures
(continuous line for our test $\Psi_\alpha$, dashed line for Higher
Criticism, dashed/dotted line for the Kolmogorov--Smirnov test and
dotted line for the test based on the variance) according to $\mu_2$,
for $\varepsilon= 0.05$ (top left), $0.15$ (top right), $0.25$ (middle
left), $0.35$ (middle right) and $0.45$ (bottom) in the Gaussian
mixture framework.}
\label{Fig:Power:unknownmean}
\end{figure}

Once again, our testing procedure appears to be competitive w.r.t. the
existing procedures, and even offers better performances in some
particular cases. As in the previous experiment, the behavior of the
Higher Criticism deteriorates w.r.t. our procedure as $\varepsilon$
increases, namely when we leave the \textit{sparse} regime to the
\textit{dense} one. In this setting, the test based on the variance is
quite competitive.\vadjust{\goodbreak}

Remark that the considered setting is not asymptotic at all since the
sample size is $100$. As explained in Section~\ref{SecVarDiscussion},
one can expect that the performances of the test based on the variance
will deteriorate in a sparse asymptotic regime. In order to illustrate
this discussion, we have compared the test based on the variance and
our procedure in a very sparse
context where $n=1000$ and $\varepsilon= 0.001$. The corresponding
values of the power are displayed in Table~\ref{table}.

\subsection{Laplace mixtures with unknown means}
Since our test $\Psi_\alpha$ is adapted for an even density function
$\phi$, a Laplace distribution is here considered: $\phi_L(x)=\frac{1}2
 \exp(-|x|)$. As in Section~\ref{subsection:simu2}, the power of
$\Psi_\alpha$ is compared with the one of Kolmogorov--Smirnov test and
Higher Criticism. Note that these two last tests are adapted as in
Section~\ref{subsection:simu2} but where $\Phi$ and $Z$ are now
associated to the Laplace distribution. The variance-based test
introduced in Section~\ref{SecVar} is also included in these simulations.

A Monte-Carlo procedure is proposed with $N=100\,000$ samples of size
$n=100$ from a mixture distribution $(1-\varepsilon)\phi(\cdot) +
\varepsilon\phi(\cdot-\mu_2)$ with $\varepsilon\in\{
0.05,0.15,0.25,0.35,0.45\}$ and $\mu\in[0,10]$. The power functions
of these testing procedures in the different scenarios are reported in
Figure~\ref{Fig:Power:Laplace}.

Apart in the case where $\varepsilon=0.05$, our test outperforms
Higher Criticism, Kolmogorov--Smirnov and variance-based tests in all other
conditions. As previously, the power of Higher Criticism is
deteriorated as $\varepsilon$ increases.

\section{Proofs}
\label{s:proofs}

\subsection{A preliminary result}
\label{s:general_result}

In this section, we provide a general result that emphasizes the
non-asymptotic performances of our testing procedure.
%
\begin{table}
\tablewidth=250pt
\tabcolsep=0pt
\caption{Comparison of the power of the variance based test (VB) and
our procedure (LMM) for $\varepsilon= 0.001$ and $n=1000$}
\label{table}
\begin{tabular*}{250pt}{@{\extracolsep{\fill}}lllll@{}}
\hline
$\mu_2$ & 2 & 4 & 6 & 8 \\
\hline
LMM & 0.0642 & 0.3006 & 0.6131 & 0.6513 \\
VB & 0.0596 & 0.1147 & 0.2445 & 0.405 \\
\hline
\end{tabular*}
\end{table}
%
\begin{figure}

\includegraphics{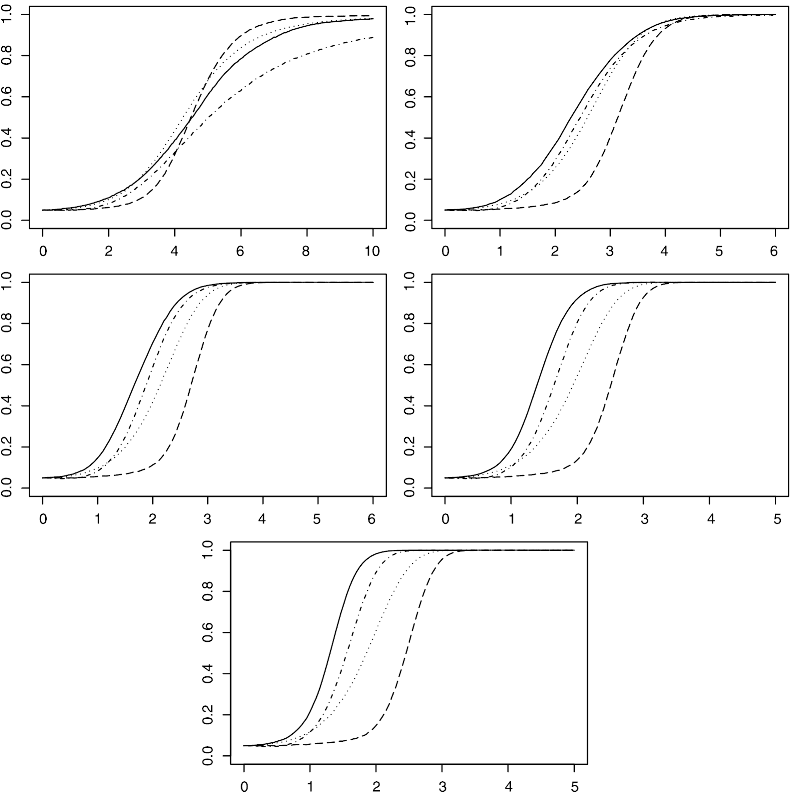}

\caption{Power function of the three considered testing procedures
(continuous line for our test $\Psi_\alpha$, dashed line for Higher
Criticism, dashed/dotted line for the Kolmogorov--Smirnov test and
dotted line for the test based on the variance) according to $\mu_2$,
for $\varepsilon= 0.05$ (top left), $0.15$ (top right), $0.25$ (middle
left), $0.35$ (middle right) and $0.45$ (bottom) in the Laplace
mixture framework.}\vspace*{-6pt}
\label{Fig:Power:Laplace}
\end{figure}

Let $\bar\Phi(x) = 1 - \Phi(x)$, where $\Phi$ is the cumulative
distribution function associated to the density function $\phi$. For
all $\alpha\in\,]\,0,1[$ and $k \in\{1,2, \ldots, n/2 \}$, let $
t_{\alpha,k}$
be a positive real number defined by
%
\begin{equation}
\label{def:talphak} \bar\Phi \biggl(\frac{t_{\alpha,k}}{2} \biggr) = \frac{k} n
\biggl[1-\sqrt{\frac{2\log(\sfrac{4} {\alpha})}{k}} \biggr]
\end{equation}
if $k>2\log(\frac{4} \alpha)$, and $t_{\alpha,k}=+\infty$ otherwise.
For all $\alpha\in\,]\,0,1[$, $\rho>0$, and $k \in\{1,2, \ldots, n/2 \}
$, we consider the subset $ \bar{\mathcal{S}}(\alpha,\rho,k)$ of
$\mathbb{R}^3$
defined by:
%
\begin{eqnarray}
\label{Srho-orderstat}
&&\bar{\mathcal{S}}(\alpha,\rho,k)\nonumber\\[-8pt]\\[-8pt]
&&\quad =\lleft\{ %
\begin{array} {l} (\varepsilon,\mu_1,\mu_2)\in\,]\,0,1[\,\times
\R^2, \mu_2 >\mu_1; \exists c\in\mathbb{R}
\mbox{ such that:}
\\
(1-\varepsilon) \bar\Phi \bigl(t_{\alpha,k} -c + \varepsilon(\mu
_2-\mu_1) \bigr) + \varepsilon\bar\Phi
\bigl(t_{\alpha,k} -c - (1- \varepsilon) (\mu _2-
\mu_1) \bigr) >\rho
\\\noalign{\vspace*{2pt}}
(1-\varepsilon) \bar\Phi \bigl(c - \varepsilon(\mu_2-
\mu_1) \bigr) + \varepsilon\bar\Phi \bigl(c + (1- \varepsilon) (
\mu_2-\mu _1) \bigr) >\rho
\end{array} %
 \rright\}.\nonumber
\end{eqnarray}

When $t_{\alpha,k}=+\infty$, we use the convention $\bar{\mathcal
{S}}(\alpha,\rho,k)= \emptyset$ for all $\rho>0$.

The following proposition highlights the non-asymptotic performances of
the test $\Psi_\alpha$.\vadjust{\goodbreak}

%
\begin{them}\label{Th:testStatOrd}
Let $\alpha\in\,]\,0,1[$ and $\beta\in\,]\,0,1-\alpha[$. Consider the test
$\Psi_{\alpha}$ described in \eqref{def:testStatOrd}.
We assume that $ n \geq8 \log(4/\alpha_n)$.
Consider the alternative sets
\[
\bar{ \F}_1[n,\alpha,\beta] = \biggl\{f(\cdot)=(1-\varepsilon)\phi (\cdot-
\mu_1) + \varepsilon\phi(\cdot-\mu_2); (\varepsilon,
\mu_1,\mu _2)\in\bigcup_{k\in\K_n}  \bar{\mathcal{S}}\bigl(\alpha_n,\rho(k,n),k\bigr) \biggr\}\vadjust{\goodbreak}
\]
where, for all $k \in \K_n$, $\bar{\mathcal{S}}(\alpha_n,\rho
(k,n),k)$ is defined by \eqref{Srho-orderstat} with
\[
\rho(k,n)=\frac{k} n + \frac{1+\sqrt{1+2k\beta}}{n\beta}. %
\]
Then $\Psi_{\alpha}$ is a level-$\alpha$ test and
\[
\sup_{f\in\bar{ \F}_1[n,\alpha,\beta]}
\P_f(\Psi _{\alpha}=0)\leq\beta. %
\]
\end{them}

In this theorem, we have defined a set $ \bar{ \F}_1[n,\alpha,\beta
]$ over which the level-$\alpha$ test statistics $\Psi_{\alpha}$
has a power greater than $1-\beta$. This result holds for all $n$,
provided that $ n \geq8 \log(4/\alpha_n)$, it is non-asymptotic. The
definition of the set $\bar{\mathcal{S}}(\alpha,\rho,k)$ is quite
rough. Nevertheless, it will allow us to describe several situations
for which
the power of our testing procedure will be assessed, in both asymptotic
and non-asymptotic cases.

The condition $ n \geq8 \log(4/\alpha_n)$ ensures that there exists
$k \in\K_n$ such that $k>2\log(4/\alpha_n)$.
Since $\alpha_n \geq\alpha/|\K_n|$, and $| \K_n| \leq\log
_2(n/2)$, this condition is satisfied if
$ n\geq8 \log(4 \log_2(n/2)/\alpha)$. For $\alpha=0.05$, this
condition holds at least for $n \geq49 $.

\subsection{Proof of Theorem \texorpdfstring{\protect\ref{Th:testStatOrd}}{6.1}}
Following the definition of $\alpha_n$, $\Psi_\alpha$ is ensured to
be a level-$\alpha$ test. In order to control the second kind error of
the test $\Psi_\alpha$, we first give an upper bound for $q_{\alpha
_n,k}$. Under the null hypothesis,
there exists $\mu\in\R$ such that $f(\cdot)=\phi(\cdot-\mu)$. Thus
$X_{(n-k+1)} - X_{(k)}$ is distributed as $Y_{(n-k+1)} - Y_{(k)}$ where
$(Y_1,\ldots,Y_{n})$
is a $n$ sample from the density $\phi(\cdot)$. Hence, if we find
$c_{\alpha_n,k}$ such that
$\P(Y_{(n-k+1)}-Y_{(k)}>c_{\alpha_n,k})\leq\alpha_n$ then
$q_{\alpha_n,k} \leq c_{\alpha_n,k}$.
For all $d \in\R$,
\[
\P(Y_{(n-k+1)} - Y_{(k)}>c_{\alpha_n,k}) \leq
\P(Y_{(n-k+1)}> c_{\alpha_n,k} + d) + \P(Y_{(k)}\leq d).
\]
According to Lemma~\ref{Append:Lemma:quant}, if $d$ fulfills $\Phi
(d)\leq\frac{k} n  [1-\sqrt{\frac{2\log(\sfrac{4}{\alpha
_n})}{k}} ]$ then $\P(Y_{(k)}\leq d)\leq\frac{\alpha_n}{2}$.
Moreover, by the same lemma, if $c_{\alpha_n,k}$ is chosen such that
$\bar\Phi(c_{\alpha_n,k}+d) \leq\frac{k} n  [1-\sqrt{\frac
{2\log(\sfrac{4}{\alpha_n})}{k}} ]$
then $\P(Y_{(n-k+1)}\geq c_{\alpha_n,k}+d)\leq\frac{\alpha_n}{2}$.
Choosing $d$ and $c_{\alpha_n,k}$ exactly such that
\[
\Phi(d)=\bar\Phi(c_{\alpha_n,k}+d)=\frac{k} n \biggl[1-\sqrt{
\frac
{2\log(\sfrac{4}{\alpha_n})}{k}} \biggr]
\]
and since $\phi(\cdot)$ is an even continuous function, we obtain that
$d=-\frac{c_{\alpha_n,k}}{2}$.
Finally, choosing $c_{\alpha_n,k} = t_{\alpha_n,k}$ where $\bar\Phi
(\frac{t_{\alpha_n,k}}{2})=\frac{k} n  [1-\sqrt{\frac{2\log
(\sfrac{4}{\alpha_n})}{k}} ]$, $\P_{H_0}(X_{(n-k+1)} - X_{(k)}>
t_{\alpha_n,k})\leq\alpha_n$
and thus $q_{\alpha_n,k} \leq t_{\alpha_n,k}$.

Considering $f\in\bar{ \F}_1[n,\alpha,\beta]$, we want to control
the second kind error of the test:
%
\begin{eqnarray}\label{adapt}
\P_f(\Psi_\alpha=0) &=& \P_f ({\forall k \in
\K _n,X_{(n-k+1)}-X_{(k)}\leq q_{\alpha_n,k}} )
\nonumber
\\[-8pt]\\[-8pt]
&\leq& \inf_{k \in\K_n}\P_f ({X_{(n-k+1)}-X_{(k)}
\leq q_{\alpha_n,k}} ) .\nonumber
\end{eqnarray}
Since $f\in\bar{ \F}_1[n,\alpha,\beta]$, there exist $\varepsilon
\in\,]\,0,1[$ and $(\mu_1,\mu_2)\in\mathbb{R}^2$, $\mu_1<\mu_2$ such that
\[
\forall x\in\R,\qquad  f(x)=(1-\varepsilon)\phi(x-\mu_1) + \varepsilon
\phi(x-\mu_2) %
\]
and for some $k \in\K_n$, there exists a real $c$ such that
$(\varepsilon,\mu_1,\mu_2)$ fulfills the two following conditions:
%
\begin{eqnarray}
\label{Cond1}(1-\varepsilon)\bar\Phi \bigl(t_{\alpha_n,k} - c + \varepsilon(\mu
_2-\mu_1) \bigr) + \varepsilon\bar\Phi
\bigl(t_{\alpha_n,k} - c - (1- \varepsilon) (\mu_2-
\mu_1) \bigr) &>& \rho(k,n),
\\
\label{Cond2}(1-\varepsilon) \bar\Phi \bigl(c - \varepsilon(\mu_2-
\mu_1) \bigr) + \varepsilon\bar\Phi \bigl(c + (1- \varepsilon) (
\mu_2-\mu _1) \bigr) &>& \rho(k,n),
\end{eqnarray}
with $\rho(k,n)=\frac{k} n + \frac{1+\sqrt{1+2k\beta}}{n\beta}$.
Using \eqref{adapt} and the fact that $ q_{\alpha_n,k} \leq t_{\alpha
_n,k} $,
%
\begin{eqnarray}\label{4ProbaControl}
\P_f ({X_{(n-k+1)}-X_{(k)}\leq q_{\alpha_n,k}} ) &
\leq&\P _f(X_{(n-k+1)}-X_{(k)}\leq t_{\alpha_n,k})
\nonumber
\\
&\leq& \P_f\bigl(X_{(n-k+1)}\leq t_{\alpha_n,k} +
\E_f[X_1] - c \bigr)
\\
& &{} +\P_f\bigl( X_{(k)}>\E_f[X_1]
- c\bigr). \nonumber
\end{eqnarray}

For the first term in the right-hand side of \eqref{4ProbaControl},
\begin{eqnarray*}
\P_f \bigl(X_{(n-k+1)} \leq t_{\alpha_n,k} +
\E_f[X_1] - c \bigr) & \leq& \P_f \Biggl(
\sum_{i=1}^n \mathds{1}_{\{X_i\leq t_{\alpha_n,k}
+\E_f[X_1] - c\}} >
n-k \Biggr)
\\
& \leq& \P_f \Biggl(\sum_{i=1}^n
\{\mathds{1}_{\{X_i\leq
t_{\alpha_n,k} +\E_f[X_1] - c \}} - q_1 \}> n (1-q_1) -k
\Biggr)
\end{eqnarray*}
with
\begin{eqnarray*}
q_1 &=&\P_f \bigl(X_1\leq
t_{\alpha_n,k} +\E_f[X_1] - c \bigr)
\\
&=& (1-\varepsilon) \Phi \bigl(t_{\alpha_n,k} +\E_f[X_1]
- c - \mu _1 \bigr) + \varepsilon\Phi \bigl(t_{\alpha_n,k} +
\E_f[X_1] - c - \mu_2 \bigr)
\\
&=& (1-\varepsilon) \Phi \bigl(t_{\alpha_n,k}-c + \varepsilon(\mu
_2-\mu_1) \bigr) + \varepsilon\Phi \bigl(t_{\alpha_n,k}
-c - (1- \varepsilon) (\mu_2-\mu_1) \bigr)
\end{eqnarray*}
since $\E_f[X_1]=(1-\varepsilon) \mu_1 + \varepsilon\mu_2$.
Condition \eqref{Cond1} gives that $n(1-q_1)-k>0$ and using Markov's
inequality,
\[
\P_f \bigl(X_{(n-k+1)} < t_{\alpha_n,k} +
\E_f[X_1] - c \bigr)\leq \frac{nq_1(1-q_1)}{[n(1-q_1) -k]^2} \leq
\frac{n(1-q_1)}{[n(1-q_1) -k]^2}. %
\]
Note that the inequality $\frac{nx}{(nx-k)^2} \leq\frac{\beta}{2}$
is fulfilled if and only if $x\notin [\frac{k} n + \frac{1}{\beta
n} \pm\frac{\sqrt{1 + 2k\beta}}{\beta n} ]$.
Then, since condition \eqref{Cond1} ensures us that $1-q_1\notin
[\frac{k} n + \frac{1}{n\beta} \pm \frac{\sqrt{1 + 2k\beta
}}{n\beta} ]$,
\[
\P_f \bigl(X_{(n-k+1)} < t_{\alpha_n,k}+
\E_f[X_1]-c \bigr) \leq \frac{\beta}{2}. %
\]
For the second term in the right-hand side of \eqref{4ProbaControl},
\[
\P_f \bigl(X_{(k)} > \E_f[X_1]-
c \bigr) \leq \P_f \Biggl(\sum_{i=1}^n
\{\mathds{1}_{\{X_i> \E_f[X_1] - c\}} - q_2 \}> n (1-q_2) -k
\Biggr) %
\]
with
\begin{eqnarray*}
q_2 &=&\P_f \bigl(X_1>
\E_f[X_1] - c \bigr)
\\
&=& (1-\varepsilon) \bar\Phi \bigl(\E_f[X_1] - c -
\mu_1 \bigr) + \varepsilon\bar\Phi \bigl(\E_f[X_1]
- c - \mu_2 \bigr)
\\
&=& (1-\varepsilon) \bar\Phi \bigl(-c +\varepsilon(\mu_2 - \mu
_1) \bigr) + \varepsilon\bar\Phi \bigl(-c - (1-\varepsilon) (\mu
_2-\mu_1) \bigr)
\\
&=& (1-\varepsilon) \Phi \bigl(c - \varepsilon(\mu_2-
\mu_1) \bigr) + \varepsilon\Phi \bigl(c + (1- \varepsilon) (
\mu_2-\mu_1) \bigr).
\end{eqnarray*}
Condition \eqref{Cond2} gives that $n(1-q_2)-k>0$ and using Markov's
inequality,
\[
\P_f \bigl(X_{(k)} > \E_f[X_1] -
c \bigr)\leq\frac{nq_2
(1-q_2)}{[n(1-q_2) -k]^2} \leq\frac{n(1-q_2)}{[n(1-q_2) -k]^2}. %
\]
According to condition \eqref{Cond2}, $1-q_2\notin [\frac{k} n +
\frac{1}{n\beta} \pm \frac{\sqrt{1 + 2k\beta}}{n\beta} ]$,
thus
\[
\P_f \bigl(X_{(k)} >\E_f[X_1] -
c \bigr) \leq\frac{\beta}{2}. %
\]
Finally, $\P_f(\Psi_\alpha=0)\leq\beta$.

\subsection{Proof of Theorem \texorpdfstring{\protect\ref{Th:lowerbound:dense}}{3.1}}
We define
\[
\F_{1,G}[\rho,M] = \bigl\{{f \in \F_{1,G}[M], \varepsilon( 1-
\varepsilon) (\mu_2-\mu_1)^2 \geq\rho} \bigr
\} . %
\]
Let $T_\alpha$ be a level-$\alpha$ test. For all $f\in\F_{1,G}[\rho,M]$,
\begin{eqnarray*}
\P_f(T_\alpha=0) &=& \P_{\phi_G}(T_\alpha=0)+
\P_f(T_\alpha=0) - \P_{\phi_G}(T_\alpha=0)
\\
&\geq& 1-\alpha- \bigl[\P_{\phi_G}(T_\alpha=0) -
\P_f(T_\alpha=0)\bigr].
\end{eqnarray*}
Thus for a density $\tilde f \in\F_{1,G}[\rho,M]$ which has to be
specified after,
\begin{eqnarray*}
\sup_{f\in\F_{1,G}[\rho,M]}  \P_f(T_\alpha=0)
&\geq& 1-\alpha- \bigl[\P_{\phi_G}(T_\alpha=0) -
\P_{\tilde f}(T_\alpha=0)\bigr]
\\
&\geq& 1-\alpha- \|\P_{\phi_G} - \P_{\tilde f}\|_{\mathrm{TV}},
\end{eqnarray*}
where $\|P - Q\|_{\mathrm{TV}}$ denotes the total variation distance between two
probability distributions $P$ and $Q$. Since
$\|\P_{\phi_G} - \P_{\tilde f}\|_{\mathrm{TV}}\leq\sqrt{2 [1 - A(\phi
_G,\tilde f)^n]}$ where $A(\phi_G,\tilde f)=\int_\R\sqrt{\phi
_G(x)\tilde f(x)}\,\mathrm{d}x$ is
the Hellinger affinity between the two density functions $\phi_G$ and~$\tilde f$,
\[
\beta\bigl(\F_{1,G}[\rho,M]\bigr):= \inf_{T_\alpha}
\sup_{f\in\F_{1,G}[\rho,M]}  \P_f(T_\alpha=0)
\geq1-\alpha- \sqrt{2 \bigl[1 - A(\phi_G,\tilde f)^n
\bigr]}. %
\]
If we specify a density $\tilde f\in\F_{1,G}[\rho,M]$ such that
$A(\phi_G,\tilde f) \geq c(\alpha,\beta)^{\sfrac{1}{ n}}$ then
$\beta(\F_{1,G}[\rho,M])\geq1 - \alpha- (1-\alpha-\beta)=\beta
$. Moreover, since
\[
A(\phi_G,\tilde f)\geq1 - \frac{1} 2 \E_\phi
\biggl[ \biggl(\frac
{\tilde f(X)-\phi_G(X)}{\phi_G(X)} \biggr)^2 \biggr],
\]
$A(\phi_G,\tilde f)\geq c(\alpha,\beta)^{\sfrac{1} n}$ is obtained if
$\E_{\phi_G} [ (\frac{\tilde f(X)-\phi_G(X)}{\phi
_G(X)} )^2 ]\leq2 [1-c(\alpha,\beta)^{\sfrac{1}{n}} ]$.

In the sequel, we consider the density $\tilde f = (1-\varepsilon)\phi
(\cdot-\mu_1) + \varepsilon\phi(\cdot-\mu_2)$, with
%
\begin{eqnarray}
\label{Cond:ftilde1}(1-\varepsilon)\mu_1 &=&-\varepsilon\mu_2,
\\
\label{Cond:ftilde2}\max\bigl(\mu_1^2,\mu_2^2,|
\mu_1\mu_2|\bigr)&\leq&\nu^2 =
\frac
{M^2}{4},
\\
\label{Cond:ftilde3}\varepsilon(1-\varepsilon) (\mu_2-\mu_1)^2
&=& \rho.
\end{eqnarray}
In particular, $\tilde f \in\F_{1,G}[\rho,M]$ since $(\mu_2-\mu
_1)^2 \leq M^2$.

For this choice,
\begin{eqnarray*}
&&\E_{\phi_G} \biggl[ \biggl(\frac{\tilde f(X)-\phi_G(X)}{\phi
_G(X)} \biggr)^2
\biggr]\\
&&\quad = \int_\R\frac{[\tilde f(x) - \phi
_G(x)]^2}{\phi_G(x)} \,\mathrm{d}x
\\
&&\quad = \int_\R\frac{\{(1-\varepsilon)[\phi_G(x-\mu_1)-\phi_G(x)] +
\varepsilon[\phi_G(x-\mu_2)-\phi_G(x)]\}^2}{\phi_G(x)} \,\mathrm{d}x
\\
&&\quad = (1-\varepsilon)^2 \biggl[\int_\R
\frac{\phi_G(x-\mu_1)^2}{\phi
_G(x)} \,\mathrm{d}x - 1 \biggr] + \varepsilon^2 \biggl[\int
_\R\frac{\phi
_G(x-\mu_2)^2}{\phi_G(x)} \,\mathrm{d}x - 1 \biggr]
\\
& &\qquad {} + 2\varepsilon(1-\varepsilon) \biggl[\int_\R
\frac{\phi
_G(x-\mu_1) \phi_G(x-\mu_2)}{\phi_G(x)} \,\mathrm{d}x - 1 \biggr].
\end{eqnarray*}
We have $\int_\R\frac{\phi_G(x-\mu_1) \phi_G(x-\mu_2)}{\phi
_G(x)} \,\mathrm{d}x = \exp(\mu_1\mu_2)$, for all $\mu_1,\mu_2 \in\mathbb
{R}$, hence 
\begin{eqnarray*}
\E_{\phi_G} \biggl[ \biggl(\frac{\tilde f(X)-\phi_G(X)}{\phi
_G(X)} \biggr)^2
\biggr] &=& (1-\varepsilon)^2 \bigl[\mathrm{e}^{\mu_1^2} -1 \bigr] +
\varepsilon^2 \bigl[\mathrm{e}^{\mu_2^2} -1 \bigr] + 2\varepsilon(1-
\varepsilon) \bigl[\mathrm{e}^{\mu_1\mu_2} -1 \bigr].
\end{eqnarray*}
Next, using that $|\mathrm{e}^u -1 -u - \frac{1} 2 u^2|\leq\frac{\mathrm{e}^{U^2}}{3!}
|u|^3$ for all $|u|<U$ with condition \eqref{Cond:ftilde2},
\begin{eqnarray*}
\E_{\phi_G} \biggl[ \biggl(\frac{\tilde f(X)-\phi_G(X)}{\phi
_G(X)} \biggr)^2
\biggr] &\leq& (1-\varepsilon)^2 \biggl[\mu_1^2
+ \frac{1} 2 \mu_1^4 + \frac
{\mathrm{e}^{\nu^2}}{3!}
\mu_1^6 \biggr]
\\
& &{} + \varepsilon^2 \biggl[\mu_2^2 +
\frac{1} 2 \mu_2^4 + \frac
{\mathrm{e}^{\nu^2}}{3!}
\mu_2^6 \biggr]
\\
& &{} + 2\varepsilon(1-\varepsilon) \biggl[\mu_1\mu_2 +
\frac{1} 2 \mu _1^2\mu_2^2
+ \frac{\mathrm{e}^{\nu^2}}{3!}|\mu_1\mu_2|^3 \biggr]
\\
&\leq& \bigl[(1-\varepsilon)\mu_1 + \varepsilon\mu_2
\bigr]^2 + \frac{1} 2 \bigl[(1-\varepsilon)
\mu_1^2 + \varepsilon\mu_2^2
\bigr]^2
\\
& &{} + \frac{\mathrm{e}^{\nu^2}}{3!} \bigl[(1-\varepsilon)|\mu_1|^3
+ \varepsilon|\mu_2|^3 \bigr]^2.
\end{eqnarray*}
The parameters of $\tilde f$ are constrained such that $(1-\varepsilon
)\mu_1 + \varepsilon\mu_2=0$ thus
\begin{eqnarray*}
&&\E_{\phi_G} \biggl[ \biggl(\frac{\tilde f(X)-\phi_G(X)}{\phi
_G(X)} \biggr)^2
\biggr] \\
&&\quad \leq \frac{1} 2 \bigl[(1-\varepsilon)\varepsilon(
\mu_2-\mu _1)^2 \bigr]^2 +
\frac{\mathrm{e}^{\nu^2}}{3!} \bigl\{(1-\varepsilon)\varepsilon|\mu _2-
\mu_1|^3 \bigl[\varepsilon^2 + (1-
\varepsilon)^2\bigr] \bigr\}^2
\\
&&\quad \leq (1-\varepsilon)^2 \varepsilon^2 (
\mu_2-\mu_1)^4 \biggl[\frac{1} 2 +
\frac{\mathrm{e}^{\nu^2}}{3!} (\mu_2-\mu_1)^2 \bigl[
\varepsilon^2 + (1-\varepsilon)^2\bigr]^2
\biggr]
\\
&&\quad \leq C^2(M) \bigl[(1-\varepsilon) \varepsilon(\mu_2-
\mu _1)^2 \bigr]^2 = C^2(M)
\rho^2
\end{eqnarray*}
with $C^2(M) = \frac{1}{2} + \frac{2}{3} M^2 \mathrm{e}^{M^2/4}$. Moreover, if
$u<0$, $1-\mathrm{e}^u\geq-u -\frac{1} 2 u^2$ thus $1-c(\alpha,\beta)^{\sfrac{1}{n}}
\geq-\frac{1} n \log c(\alpha,\beta) - \frac{1}{2}  (\frac
{\log c(\alpha,\beta)}{n} )^2$.
Then, the condition
\[
\rho= (1-\varepsilon) \varepsilon(\mu_2-\mu_1)^2
< \frac{1}{C(M)} \sqrt{ -\frac{2} n \log c(\alpha,\beta) - \biggl(
\frac{\log c(\alpha
,\beta)}{n} \biggr)^2 }:=\rho^\star %
\]
implies that $\beta(\F_{1,G}[\rho,M]) > \beta$.

\subsection{Proof of Theorem \texorpdfstring{\protect\ref{Th:majo}}{3.2}}

Let $f(\cdot)=(1-\varepsilon) \phi_G(\cdot-\mu_1) + \varepsilon\phi
_G(\cdot-\mu_2) \in\F_{1,G}[\rho,M] $ where $\rho$ satisfies \eqref
{Conddense}.
We will prove that $f \in\bar{\mathcal{F}}_1[n,\alpha,\beta]$ and
the result will be a consequence of Theorem~\ref{Th:testStatOrd}.
In the following, we consider
$k \in\K_n$ such that
\[
\frac{0.99}{2} \bar\Phi_G(M)\leq\frac{k}{n} \leq 0.99
\bar\Phi _G(M) .
\]
Note that this is possible since, under the assumptions of Theorem~\ref
{Th:majo}, $ 0.99 \bar\Phi_G(M) n \geq1$.
Note that
$|\K_n| \leq\log_2(n/2)$, hence $\alpha_n \geq\alpha/|\K_n| \geq
\alpha/\log_2(n/2)$. We will show that $(\varepsilon, \mu_1,\mu
_2)\in\bar{\mathcal{S}}(\alpha_n,\rho(k,n),k)$:
Considering $c= {t_{\alpha_n,k}}/{2}$ and denoting $\tau=\mu_2-\mu
_1$, we want to prove that
%
\begin{eqnarray}
\label{E1}(1-\varepsilon) \bar\Phi_G \biggl(\frac{t_{\alpha_n,k}}{2} + \varepsilon
\tau \biggr) + \varepsilon\bar\Phi_G \biggl(\frac{t_{\alpha_n,k}}{2} - (1-
\varepsilon) \tau \biggr) &>& \rho(k,n),
\\
\label{E2}(1-\varepsilon) \bar\Phi_G \biggl(\frac{t_{\alpha_n,k}}{2} - \varepsilon
\tau \biggr) + \varepsilon\bar\Phi_G \biggl(\frac{t_{\alpha_n,k}}{2} + (1-
\varepsilon) \tau \biggr) & >& \rho(k,n)
\end{eqnarray}
hold, with $\rho(k,n)=\frac{k} n + \frac{1}{n\beta} + \frac{\sqrt{1
+ 2k\beta}}{n\beta}$.

We use a Taylor expansion at the order 2, the terms of order 1 vanish
and this leads to:
\begin{eqnarray*}
&&(1-\varepsilon) \bar\Phi_G \biggl(\frac{t_{\alpha_n,k}}{2} +\varepsilon
\tau \biggr) + \varepsilon\bar\Phi_G \biggl(\frac
{t_{\alpha_n,k}}{2} - (1-
\varepsilon)\tau \biggr)
\\
&&\quad  = \bar\Phi_G \biggl(\frac{t_{\alpha_n,k}}{2} \biggr) +
\frac{1} 2 (1-\varepsilon) \varepsilon\tau^2 \bigl[\varepsilon
\bigl(-\phi_G'(a)\bigr) + (1-\varepsilon) \bigl(-
\phi_G'(b)\bigr) \bigr],
\end{eqnarray*}
where $a$ (resp. $b$) belongs to the interval $ ]\frac{t_{\alpha
_n,k}}{2},\frac{t_{\alpha_n,k}}{2}+\varepsilon\tau [$
(resp. $ ]\frac{t_{\alpha_n,k}}{2}-(1-\varepsilon)\tau,\frac
{t_{\alpha_n,k}}{2} [$).

We recall that $\bar\Phi_G (\frac{t_{\alpha_n,k}}{2} ) =
\frac{k}{n}  [1 - \sqrt{\frac{2 \log( 4 /\alpha_n)}{k}} ]$.
Hence, in order to prove that \eqref{E1} holds, we just have to show that
%
\begin{equation}
\label{eq:anbn} (1-\varepsilon) \varepsilon\tau^2 \bigl\{\varepsilon
\bigl[-\phi_G'(a)\bigr] + (1-\varepsilon) \bigl[-
\phi_G'(b)\bigr] \bigr\} \geq\frac{2}{n\beta} +
\frac{\sqrt{k}}{n} \sqrt{2\log(4 /\alpha_n)}.
\end{equation}
Next, we want to prove that $ [ \frac{t_{\alpha
_n,k}}{2}-(1-\varepsilon)\tau, \frac{t_{\alpha_n,k}}{2}+\varepsilon
\tau ]$
remains included in a fixed interval $[c_1(M),c_2(M)]$ with $c_1(M)>0$.

On one hand, we have
\[
\frac{t_{\alpha_n,k}}{2} \geq \bar\Phi_G^{-1} \biggl(
\frac{k}{n} \biggr) \geq\bar\Phi_G^{-1} \bigl( 0.99
\bar\Phi_G(M) \bigr) %
\]
and
\[
\frac{t_{\alpha_n,k}}{2} -M \geq \bar\Phi_G^{-1} \bigl( 0.99
\bar \Phi_G(M) \bigr) -M :=c_1(M)>0.
\]
Moreover,
\begin{eqnarray*}
\bar\Phi_G \biggl(\frac{t_{\alpha_n,k}}{2} \biggr) &\geq&
\frac
{0.99}{2} \bar\Phi_G(M) - \sqrt{\frac{2 \log( 4 /\alpha_n)}{\sqrt{n}}}\sqrt{
0.99 \bar\Phi _G(M) }
\\
&\geq& \frac{\bar\Phi_G(M)}{200}
\end{eqnarray*}
since $(8.25){\log(4\log_2(n/2)/\alpha)}/{n} \leq\bar\Phi_G(M)$.
This implies that
\[
\frac{t_{\alpha_n,k}}{2}+\tau\leq \bar\Phi_G^{-1} \biggl(\frac{ \bar\Phi
_G(M)}{200} \biggr) +M :=c_2(M).
\]

Finally, the function $ -\phi_G' $ is bounded from below on this
interval by some positive constant
$C(M)= \min_{x\in[c_1(M),c_2(M)]}( -\phi_G'(x))$. This implies that
\eqref{eq:anbn} is satisfied if
$ \varepsilon(1-\varepsilon) \tau^2 \geq C(\alpha,\beta,M) {\sqrt
{\log\log(n)}}/{\sqrt{n}}$ for some suitable constant $C(\alpha,
\beta,M)$.
This concludes the proof of \eqref{E1}. The proof of \eqref{E2}
follows the same arguments.

\begin{Remark*} If we choose $k^* \in\K_n$ such that
\[
\frac{0.99}{2} \bar\Phi_G(M)\leq\frac{k^*}{n} \leq 0.99
\bar\Phi_G(M) %
\]
and consider the test statistics
\[
\mathds{1}_{X_{(n-k^*+1)} - X_{(k^*)} > q_{\alpha,k^*}}
\]
then it is easy to prove that \eqref{E1} and \eqref{E2} are satisfied
for $k=k^*$ if
$ \varepsilon(1-\varepsilon) \tau^2 \geq C'(\alpha,\beta,M)
/{\sqrt{n}}$ for some suitable constant $C'(\alpha, \beta,M)$ since
in this case
$\alpha_n$ is replaced by $\alpha$ and we do no more have the
logarithmic loss in the rate of convergence.
\end{Remark*}

\subsection{Proof of Proposition \texorpdfstring{\protect\ref{Prop:upperbound:dense}}{3.1}}
Following the definition of the threshold $v_{\alpha,n}$, it is easy
to see that $\psi_\alpha$ defined in \eqref{eq:test_variance} is a
level-$\alpha$ test. Now, our aim is to upper bound the term
\[
\P_f( \psi_\alpha= 0) = \P_f
\bigl(S_n^2 \leq v_{\alpha,n}\bigr)
\]
when $f\in\F_1[\rho,M]$ where, as previously,
\[
\F_{1}[\rho,M] = \bigl\{{f \in \F_{1}[M], \varepsilon(
1- \varepsilon) (\mu_2-\mu_1)^2 \geq\rho}
\bigr\} . %
\]
In a first time, a control of $v_{\alpha,n}$ is required. If a real
number $c_{\alpha,n}$ is determined such that\linebreak[4]  $\P
_{H_0}(S_n^2>c_{\alpha,n})\leq\alpha$, then $v_{\alpha,n} \leq
c_{\alpha,n}$.
According to \cite{Wilks}, page 200, if
$Y_1,\dots, Y_n$ are i.i.d. random variables such that $\mathbb
{E}[(Y_1-\mathbb{E}[Y_1])^4]<+\infty$, then
%
\begin{equation}\label{lemme-Wilks}
\Var \Biggl( \frac{1}{n-1} \sum_{i=1}^n
(Y_i - \bar Y_n)^2 \Biggr) \leq
\frac{1} n \biggl\{\E\bigl[\bigl(Y_1-\E[Y_1]
\bigr)^4\bigr] - \frac
{n-3}{n-1} \Var(Y_1)^2
\biggr\}.
\end{equation}
Hence, since $\E_\phi[X_1^4]<B$ and $ \E_\phi[ S_n^2] = \sigma^2$,
\[
\P_{H_0}\bigl(S_n^2 > c_{\alpha,n}\bigr) =
\P_{H_0}\bigl(S_n^2 -\sigma^2 >
c_{\alpha,n} -\sigma^2 \bigr) \leq \frac{\Var_{\phi
}(S_n^2)}{(c_{\alpha,n}-\sigma^2)^2} \leq
\frac{B}{n (c_{\alpha,n}-\sigma^2)^2}. %
\]
In particular $ \P_{H_0}(S_n^2 > c_{\alpha,n}) \leq \alpha $
with $c_{\alpha,n}= \sigma^2 + \sqrt{\frac{B}{n\alpha}}$, and thus
\[
v_{\alpha,n} \leq\sigma^2 + \sqrt{\frac{B}{n\alpha}}.
\]
Note that $\mathbb{E}_f[S_n^2] = \Var_f(X_1) = \sigma
^2+\varepsilon(1-\varepsilon)(\mu_2-\mu_1)^2$. Hence, for all $f\in
\F_1[\rho,M]$,
\begin{eqnarray*}
\P_{f}( \psi_\alpha= 0) & \leq& \P_{f} \biggl(
S_n^2 \leq\sigma^2 + \sqrt{
\frac{B}{n\alpha
}} \biggr)
\\
& = & \P_{f} \biggl(S_n^2 -
\mathbb{E}_f\bigl[S_n^2\bigr] \leq
\sigma^2 + \sqrt {\frac{B}{n\alpha}} - \mathbb{E}_f
\bigl[S_n^2\bigr] \biggr)
\\
& \leq& \P_{f} \biggl( \bigl|S_n^2 -
\mathbb{E}_f\bigl[S_n^2\bigr]\bigr| \geq
\varepsilon (1-\varepsilon) (\mu_2-\mu_1)^2 -
\sqrt{\frac{B}{n\alpha}} \biggr)
\\
& \leq& \frac{\Var_f(S_n^2)}{ [\varepsilon(1-\varepsilon)(\mu
_2-\mu_1)^2 - \sqrt{\afrac{B}{n\alpha}} ]^2}
\end{eqnarray*}
if $\varepsilon(1-\varepsilon)(\mu_2-\mu_1)^2 > \sqrt{\frac
{B}{n\alpha}}$. Using equation \eqref{lemme-Wilks}, we get
\[
\P_{f}( \psi_\alpha= 0) \leq\frac{\mathbb{E}_{f}[(X_1-\mathbb
{E}_f[X_1])^4]}{n [\varepsilon(1-\varepsilon)(\mu_2-\mu_1)^2 -
\sqrt{\afrac{B}{n\alpha}} ]^2} .
\]
In order to conclude, just remark that
\begin{eqnarray*}
\E_f\bigl[\bigl(X_1 - \E[X_1]
\bigr)^4\bigr] &=& (1-\varepsilon) \int_\R
\bigl[x-(1-\varepsilon) \mu_1 - \varepsilon\mu_2
\bigr]^4 \phi(x-\mu_1) \,\mathrm{d}x
\\
& &{} + \varepsilon\int_\R\bigl[x-(1-\varepsilon)
\mu_1 - \varepsilon\mu _2\bigr]^4 \phi(x-
\mu_2) \,\mathrm{d}x
\\
&=& (1-\varepsilon) \int_\R\bigl[y-\varepsilon(
\mu_2 -\mu_1)\bigr]^4 \phi (y) \,\mathrm{d}y
\\
& &{} + \varepsilon\int_\R\bigl[y+(1-\varepsilon) (
\mu_2 -\mu_1)\bigr]^4 \phi (y) \,\mathrm{d}y
\\
&=& \E_\phi\bigl[Z^4\bigr] + 6 \varepsilon(1-\varepsilon)
(\mu_2 -\mu_1)^2 \E_\phi
\bigl[Z^2\bigr]\\
&&{} + \bigl[\varepsilon(1-\varepsilon)^4 +
\varepsilon^4 (1-\varepsilon)\bigr](\mu_2 -
\mu_1)^4
\\
&\leq& B + \frac{6} 4 \sqrt{B} M^2 + M^4 \leq
\bigl(M^2 + \sqrt{B}\bigr)^2.
\end{eqnarray*}
%
Thus
\[
\P_f( \psi_\alpha= 0) \leq\frac{(M^2 + \sqrt{B})^2}{n
[\varepsilon(1-\varepsilon)(\mu_2-\mu_1)^2 - \sqrt{\afrac
{B}{n\alpha}} ]^2} \leq\beta
\]
as soon as
\[
\varepsilon(1-\varepsilon) (\mu_2-\mu_1)^2
\geq\frac{C(\alpha
,\beta,M,B)}{\sqrt{n}},
\]
for some positive constant $C(\alpha,\beta,M,B)$. This concludes the
proof of Proposition~\ref{Prop:upperbound:dense}.

\subsection{Proof of Theorem \texorpdfstring{\protect\ref{Prop:sparse:statord}}{4.1}}
\label{s:preuve_sparse_gaussien}
We will prove that, under the assumptions of Theorem~\ref
{Prop:sparse:statord}, $f \in\bar{\mathcal{F}}_1[n,\alpha,\beta]$
and the result will be a consequence of Theorem~\ref{Th:testStatOrd}.
We recall that
$|\K_n| \leq\log_2(n)$, hence $\alpha\geq\alpha_n \geq\alpha
/|\K_n| \geq
\alpha/\log_2(n)$. We set $\tau=\mu_2-\mu_1$ and we have to prove
that there exists $k\in\K_n$ and $c \in\mathbb{R}$ such that
%
\begin{eqnarray}
 \label{EE1}(1-\varepsilon) \bar\Phi_G ({t_{\alpha_n,k}} -c+ \varepsilon \tau
) + \varepsilon\bar\Phi_G \bigl({t_{\alpha_n,k}} - c-(1-
\varepsilon) \tau \bigr) &>& \rho(k,n),
\\
\label{EE2}(1-\varepsilon) \bar\Phi_G ( c - \varepsilon\tau ) + \varepsilon
\bar\Phi_G \bigl(c+ (1- \varepsilon) \tau \bigr) &>& \rho(k,n),
\end{eqnarray}
with $\rho(k,n)=\frac{k} n + \frac{1}{n\beta} + \frac{\sqrt{1 +
2k\beta}}{n\beta}$. Note that
$ \rho(k,n) \leq\frac{k} n + C_{\beta} \frac{\sqrt{k}}{n}$ with
$C_{\beta}= \frac{2}{\beta}+\sqrt{\frac{2}{\beta}}$.
We recall that $ t_{\alpha_n,k} $ is defined by
\[
\bar\Phi_G \biggl(\frac{t_{\alpha_n,k}}{2} \biggr) = \frac{k} n
\biggl[1-\sqrt{\frac{2\log( 4/\alpha_n)}{k}} \biggr].
\]
In the following, we set $C_{\alpha_n}=\sqrt{2\log( 4/\alpha_n)}$.
Since $\alpha_n \geq
\alpha/\log_2(n)$, note that $0<C_{\alpha_n} \leq C(\alpha) \sqrt
{\log\log(n)}$ for some constant $C(\alpha)$ depending
only on $\alpha$. We choose $k \in\K_n$ such that
%
\begin{equation}
\label{Condk} \lim_{n\rightarrow+\infty} \frac{k}{\log(n)\log\log(n)} =
 +\infty\quad \mbox{and}\quad  \lim_{n\rightarrow+\infty} \frac{n}{k} = +\infty
\end{equation}
and we define
%
\begin{equation}
\label{defc} c= \frac{t_{\alpha_n,k}}{2} - \sqrt{\frac{2} k}
C_{\alpha_n}.
\end{equation}
For the sake of simplicity, we omit the dependency with respect to $n$
in the notation of $k$ and $c$.
Let us first show that \eqref{EE2} holds for $n$ large enough. First,
note that
\[
(1-\varepsilon) \bar\Phi_G ( c - \varepsilon\tau ) + \varepsilon
\bar\Phi_G \bigl(c+ (1- \varepsilon) \tau \bigr) > (1-\varepsilon)
\bar\Phi_G (c) . %
\]
With the assumptions on $k$, we have that $c>0$ for $n$ large enough
since $t_{\alpha_n,k} \rightarrow +\infty$ and
$ C_{\alpha_n}/\sqrt{k} \rightarrow0$ as $n \rightarrow+\infty$.
Hence
\[
\bar\Phi_G (c) \geq \bar\Phi_G \biggl(
\frac{t_{\alpha
_n,k}}{2} \biggr) + \sqrt{\frac{2} k}C_{\alpha_n}
\phi_G \biggl(\frac
{t_{\alpha_n,k}}{2} \biggr) .
\]
%
Moreover, for all $u>0$,
\[
\bar\Phi_G (u) \leq\frac{1}{2}\exp\bigl(-u^2/2
\bigr)=\sqrt{\frac{\pi}{2}} \phi_G(u),
\]
hence
\[
\phi_G \biggl(\frac{t_{\alpha_n,k}}{2} \biggr)\geq\sqrt{
\frac
{2}{\pi}} \bar\Phi_G \biggl(\frac{t_{\alpha_n,k}}{2} \biggr).
\]
This leads to
\[
(1-\varepsilon) \bar\Phi_G (c) > (1-\varepsilon) \biggl({1 +
\frac
{2 C_{\alpha_n} }{\sqrt{\pi k}}} \biggr) \bar\Phi_G \biggl(\frac
{t_{\alpha_n,k}}{2}
\biggr) .
\]
After some obvious computations, condition \eqref{EE2} is satisfied as
soon as
\[
(1-\varepsilon) C_{\alpha_n} \biggl({\frac{2}{\sqrt{\pi}}-1} \biggr)
\frac{\sqrt{k}}{n}> \varepsilon\frac{k}{n} + C_{\beta}
\frac
{\sqrt{k}}{n} + \frac{2 C_{\alpha_n} ^2 }{\sqrt{\pi} n}.
\]
Since $\varepsilon<1/\sqrt{n}$ and $k\leq n$, we have $\varepsilon k
<\sqrt{k}$. We recall that $C_{\alpha_n} \rightarrow+\infty$ as
$n \rightarrow+\infty$ and with the assumptions on $k$, we have that
$ \sqrt{k} /C_{\alpha_n} \rightarrow+\infty$ as
$n \rightarrow+\infty$, and the above inequality holds for $n$ large enough.

It remains to prove that \eqref{EE1} is satisfied with the conditions
on $k$ imposed by \eqref{Condk} and the value of $c$ defined by \eqref{defc}.
Let $\Delta$ satisfy $0<r<\Delta\leq1$, we choose $k \in\K_n $
satisfying \eqref{Condk} and such that $n^{1-\Delta}\leq k \leq2
n^{1-\Delta}\log^2(n)$. Note that such values of $k$ exist for $n$
large enough.
It follows from Lemma~\ref{Lemma:MinBarPhi} that $ t_{\alpha_n,k}/2
\leq\sqrt{2 \Delta\log(n)}$.
First,
\begin{eqnarray*}
\bar\Phi_G ({t_{\alpha_n,k}} - c + \varepsilon\tau ) &=& \bar
\Phi_G \biggl(\frac{t_{\alpha_n,k}}{2} + \sqrt{\frac{2} k}
C_{\alpha_n} + \varepsilon\tau \biggr)
\\
&\geq& \bar\Phi_G \biggl(\frac{t_{\alpha_n,k}}{2} \biggr)- \biggl({\sqrt{
\frac{2} k} C_{\alpha_n} + \varepsilon\tau} \biggr) \phi
_G \biggl(\frac{t_{\alpha_n,k}}{2} \biggr)
\\
&\geq& \frac{k} n \biggl[1-\frac{C_{\alpha_n}}{\sqrt{k}} \biggr] - \biggl({\sqrt{
\frac{2} k} C_{\alpha_n} + \varepsilon\tau} \biggr)
\phi_G \biggl(\frac{t_{\alpha_n,k}}{2} \biggr).
\end{eqnarray*}
We have to give an upper bound for $ \phi_G (\frac{t_{\alpha
_n,k}}{2} )$. We use the inequality
\[
\forall u >0 ,\qquad  \bar\Phi_G(u) \geq \biggl({\frac{1}{u}-
\frac
{1}{u^3} } \biggr) \phi_G(u),
\]
this leads to
\[
\forall u >0 , \qquad \phi_G(u) \leq\frac{u^3}{u^2-1} \bar
\Phi_G(u) \leq u^3 \bar\Phi_G(u) ,
\]
provided that $u^2-1 \geq1$. This is the case, for $n$ large enough
for $ u= t_{\alpha_n,k}/2 $, hence we have
\begin{eqnarray*}
\phi_G \biggl(\frac{t_{\alpha_n,k}}{2} \biggr) &\leq& \biggl[
\frac
{t_{\alpha_n,k}}{2} \biggr]^3 \bar\Phi_G \biggl(
\frac{t_{\alpha
_n,k}}{2} \biggr)
\\
&\leq& \bigl[\sqrt{2 \Delta\log(n)} \bigr]^3 \frac{k} n
\\
&\leq& 4 \sqrt{2} \bigl[\log(n) \bigr]^{7/2} n^{-\Delta}.
\end{eqnarray*}
Finally, we obtain that
\[
\bar\Phi_G ({t_{\alpha_n,k}} - c + \varepsilon\tau ) \geq
\frac{k} n - C_{\alpha_n}\frac{\sqrt{k}}{n} - \biggl({
\frac
{\sqrt{2} C_{\alpha_n}}{\sqrt{k}} + \varepsilon\tau} \biggr) 4 \sqrt{2} \bigl[\log(n)
\bigr]^{7/2} n^{-\Delta}. %
\]
Second, we want to lower bound $\bar\Phi_G ({t_{\alpha_n,k}} -
c-(1- \varepsilon) \tau )$. We have that
\begin{eqnarray*}
\bar\Phi_G \bigl({t_{\alpha_n,k}} - c-(1- \varepsilon) \tau
\bigr) &=&\bar\Phi_G \biggl(\frac{t_{\alpha_n,k}}{2}+ \sqrt{
\frac{2} k} C_{\alpha_n}-(1- \varepsilon) \tau \biggr)
\\
&\geq& \bar\Phi_G \biggl( \sqrt{2 \Delta\log(n)} -\tau+ \sqrt {
\frac{2}{k}} C_{\alpha_n} + \varepsilon\tau \biggr)
\\
&\geq& \bar\Phi_G \bigl({\sqrt{2 \Delta\log(n)}- \sqrt{2 r \log
(n)} } \bigr)
\\
& &{} - \biggl({\varepsilon\tau+ \frac{\sqrt{2} C_{\alpha_n}}{\sqrt
{k}}} \biggr) \phi_G
\bigl(\sqrt{2 \Delta\log(n)}- \sqrt{2 r \log(n)} \bigr)
\end{eqnarray*}
since $\tau=\sqrt{2 r \log(n)}$. Moreover, since
$ \phi_G(\sqrt{2 \Delta\log(n)}- \sqrt{2 r \log(n)} ) = (\sqrt
{2\pi})^{-1} n^{-(\sqrt{\Delta}-\sqrt{r})^2}$, and using again the
inequality
$\bar\Phi_G (u) \geq(\frac{1}{u}- \frac{1}{u^3}) \phi_G(u) $
which holds for all $u>0$, we obtain that
\[
\bar\Phi_G \bigl({t_{\alpha_n,k}} - c-(1- \varepsilon) \tau
\bigr) \geq C n^{-(\sqrt{\Delta}-\sqrt{r})^2} \biggl({\frac{1}{\sqrt
{\log(n)}}- \varepsilon\tau-
\frac{\sqrt{2} C_{\alpha_n}}{\sqrt
{k}}} \biggr),
\]
for some positive constant $C$ depending on $\Delta$ and $r$.
Condition \eqref{EE1} is thus fulfilled if
\begin{eqnarray*}
&&C\varepsilon n^{-(\sqrt{\Delta}-\sqrt{r})^2} \biggl({\frac
{1}{\sqrt{\log(n)}} -\varepsilon\tau-
\sqrt{\frac{2} k} C_{\alpha
_n} } \biggr)\\
&&\qquad > \varepsilon
\frac{k}{n}+ (C_{\alpha_n}+C_{\beta}) \frac{\sqrt{k}}{n} +
\biggl({\sqrt{\frac{2} k} C_{\alpha_n} + \varepsilon\tau} \biggr) 4
\sqrt{2} \bigl[\log(n) \bigr]^{7/2} n^{-\Delta}. %
\end{eqnarray*}
By \eqref{Condk}, $ C_{\alpha_n}/\sqrt{k} =\mathrm{o}(1/\sqrt{\log(n)})$,
and the left-hand side of this inequality is equivalent as $n
\rightarrow+\infty$ to $ C \varepsilon n^{-(\sqrt{\Delta}-\sqrt
{r})^2}/ \sqrt{\log(n)}$ and the right-hand side is equivalent as $n
\rightarrow+\infty$ to
$ 8 C_{\alpha_n}  ({\log(n)} )^{7/2} n^{-\Delta}/\sqrt
{k} $. Hence, the condition \eqref{EE1} will be satisfied
asymptotically if for some $\Delta\in\,]\,0,1]$,
\[
\delta+ (\sqrt{\Delta}-\sqrt{r})^2 < \frac{1+\Delta}{2}.
\]
\begin{itemize}
\item If $\frac{1}{2} <\delta\leq\frac{3}{4}$ and $ 0<r\leq\frac
{1}{4}$, we set $\Delta=4r$ and the above condition becomes $r>\delta
-\frac{1}{2}$.
\item If $\frac{1}{2} <\delta\leq\frac{3}{4}$ and $ r > \frac
{1}{4}$, the above condition is satisfied with $\Delta=1$ and no
additional condition is required.
\item If $\delta> \frac{3}{4}$, we set $\Delta=1$ and the above
condition becomes $r>(1-\sqrt{1-\delta})^2$.
\end{itemize}

This concludes the proof of Theorem~\ref{Prop:sparse:statord}.

\subsection{Proof of Theorem \texorpdfstring{\protect\ref{Prop:sparse:Laplace}}{4.2}}
\label{s:preuve_sparse_laplace}

We first provide an upper bound for the quantile $q_{\alpha_n,k}$ for
all $k\in\lbrace1,\dots, n/2 \rbrace$. We have seen in the proof of
Theorem~\ref{Th:testStatOrd} that
\[
q_{\alpha_n,k} \leq t_{\alpha_n,k},
\]
where
%
\begin{equation}\label
{eq:control_t0}
\bar\Phi_L \biggl( \frac{t_{\alpha_n,k}}{2} \biggr) = \frac{k}{n}
\biggl( 1 - \sqrt{ \frac{2\log(4/\alpha_n)}{k} } \biggr).
\end{equation}
This leads to
\[
\frac{1}{2} \mathrm{e}^{-\sfrac{t_{\alpha_n,k}}{2}} = \frac{k}{n} \biggl( 1 - \sqrt{
\frac{2\log(4/\alpha_n)}{k} } \biggr).
\]
Hence,\vspace*{2pt}
%
\begin{equation}\label{eq:control_t}
\frac{t_{\alpha_n,k}}{2} = \log \biggl( \frac{n}{k} \biggr) - \log \biggl( 1 -
\sqrt{ \frac{2\log(4/\alpha_n)}{k} } \biggr) - \log(2).
\end{equation}
Then, applying Theorem~\ref{Th:testStatOrd} with $c=t_{\alpha
_n,k}/2$, we get that if, for some $k\in\K_n $,\vspace*{2pt}
%
\begin{eqnarray}\label{eq:condition1}
&&(1-\varepsilon) \bar\Phi_L \biggl( \frac{t_{\alpha_n,k}}{2} +
\varepsilon(\mu_2-\mu_1) \biggr) + \varepsilon\bar
\Phi_L \biggl( \frac{t_{\alpha_n,k}}{2} - (1-\varepsilon) (
\mu_2-\mu_1) \biggr)\nonumber \\[-7pt]\\[-7pt]
&&\quad > \frac{k}{n} +
\frac{1+\sqrt{1+2k\beta}}{n\beta} \nonumber
\end{eqnarray}
and\vspace*{2pt}
\begin{eqnarray*}
&&(1-\varepsilon) \bar\Phi_L \biggl(\frac{t_{\alpha_n,k}}{2} - \varepsilon(
\mu_2-\mu_1) \biggr) + \varepsilon\bar\Phi_L
\biggl(\frac{t_{\alpha_n,k}}{2} + (1- \varepsilon) (\mu_2-
\mu_1) \biggr)\\[1.5pt]
&&\quad  > \frac{k}{n} + \frac{1+\sqrt{1+2k\beta}}{n\beta},
\end{eqnarray*}
then our test is powerful. For the sake of convenience, we will
concentrate our attention to the first inequality, the control of the
second one following essentially the same lines.

From now on, we will only deal with possible values of $k$ satisfying\vspace*{2pt}
%
\begin{equation}
\frac{t_{\alpha_n,k}}{2} > \mu_2 - \mu_1. \label{eq:condition_tau_t}
\end{equation}
Using the properties of the Laplace distribution and the equation
\eqref{eq:condition_tau_t}, the condition \eqref{eq:condition1} becomes\vspace*{1.5pt}
\begin{eqnarray*}
& & (1-\varepsilon)\times\frac{1}{2} \mathrm{e}^{ -\sklfrac{t_{\alpha_n,k}}{2}
- \varepsilon(\mu_2-\mu_1)} + \varepsilon\times
\frac{1}{2} \mathrm{e}^{-
\sklfrac{t_{\alpha_n,k}}{2} + (1-\varepsilon)(\mu_2-\mu_1) } > \frac
{k}{n}+ \frac{1+\sqrt{1+2k\beta}}{n\beta}
\\[1.5pt]
&&\quad  \Leftrightarrow\quad  \varepsilon\times\frac{1}{2} \mathrm{e}^{- \sfrac
{t_{\alpha_n,k}}{2} + (1-\varepsilon)(\mu_2-\mu_1) } >
\frac
{k}{n}+ \frac{1+\sqrt{1+2k\beta}}{n\beta} \\[1.5pt]
&&\hphantom{\quad  \Leftrightarrow\quad  \varepsilon\times\frac{1}{2} \mathrm{e}^{- \sfrac
{t_{\alpha_n,k}}{2} + (1-\varepsilon)(\mu_2-\mu_1) } >}{}- (1-\varepsilon)\times \frac{1}{2}
\mathrm{e}^{ -\sklfrac{t_{\alpha_n,k}}{2} - \varepsilon(\mu_2-\mu
_1)}
\\[1.5pt]
&&\quad  \Leftrightarrow\quad  \varepsilon\times\frac{1}{2} \mathrm{e}^{- \sklfrac
{t_{\alpha_n,k}}{2}+ (1-\varepsilon)(\mu_2-\mu_1) } >
\frac{k}{n}+ \frac{1+\sqrt{1+2k\beta}}{n\beta} \\[1.5pt]
&&\hphantom{\quad  \Leftrightarrow\quad  \varepsilon\times\frac{1}{2} \mathrm{e}^{- \sklfrac
{t_{\alpha_n,k}}{2}+ (1-\varepsilon)(\mu_2-\mu_1) } >}{}- (1-\varepsilon) \phi_L
\biggl( -\frac{t_{\alpha_n,k}}{2} \biggr)\times \mathrm{e}^{ - \varepsilon(\mu_2-\mu_1)}.
\end{eqnarray*}
Since $\phi_L(x)=\bar\Phi_L(x)$ for all $x\geq0 $ and thanks to
\eqref{eq:control_t0}, we get that
\begin{eqnarray*}
& & (1-\varepsilon)\times\frac{1}{2} \mathrm{e}^{ -\sklfrac{t_{\alpha_n,k}}{2}
- \varepsilon(\mu_2-\mu_1)} + \varepsilon\times
\frac{1}{2} \mathrm{e}^{-
\sklfrac{t_{\alpha_n,k}}{2} + (1-\varepsilon)(\mu_2-\mu_1) } > \frac
{k}{n}+ \frac{1+\sqrt{1+2k\beta}}{n\beta}
\\
& &\quad \Leftrightarrow\quad  \varepsilon\times\frac{1}{2} \mathrm{e}^{- \sklfrac
{t_{\alpha_n,k}}{2}+ (1-\varepsilon)(\mu_2-\mu_1) } >
\frac{k}{n}+ \frac{1+\sqrt{1+2k\beta}}{n\beta}- (1-\varepsilon) \frac{k} n \biggl[1-\sqrt{\frac{2\log(\sfrac
{4}{\alpha_n})}{k}}
\biggr]
\\
& &\hphantom{\quad \Leftrightarrow\quad  \varepsilon\times\frac{1}{2} \mathrm{e}^{- \sklfrac
{t_{\alpha_n,k}}{2}+ (1-\varepsilon)(\mu_2-\mu_1) } >
\frac{k}{n}+ \frac{1+\sqrt{1+2k\beta}}{n\beta}-}{}  \times \bigl( 1 - \varepsilon(\mu_2-\mu_1) +
V_n \bigr),
\end{eqnarray*}
where $V_n \leq C\varepsilon^2(\mu_2-\mu_1)^2 $ for some $C>0$. As
in the proof of Theorem~\ref{Th:testStatOrd}, we will deal with values
of $k$ having the parametrization $k/n = n^{-\Delta}$ for some $\Delta
\in\,]\,0,1[$. In particular,
\[
\sqrt{k} = n^{\vfrac{1-\Delta}{2}} \quad \mbox{and} \quad \frac{\sqrt{k}}{n} = n^{-\vfrac{1+\Delta}{2}}.
\]
A short investigation of the asymptotics of the term in the right-hand
side of the previous inequality indicates that the dominating term is
of order $\sqrt{k}/n$. Indeed, thanks to the parametrization of $k$,
$\varepsilon$ an $\mu_2-\mu_1$, we get that
\[
\varepsilon\frac{k}{n}(\mu_2-\mu_1) = \mathrm{o} \biggl(
\frac{\sqrt{k}}{n} \biggr)\quad  \mbox{and}\quad  \frac{1}{n} =\mathrm{o} \biggl(
\frac{\sqrt{k}}{n} \biggr) \quad \mbox{as } n\rightarrow+\infty.
\]
Hence, in order to guarantee that our test is powerful, we have to
ensure that
%
\begin{eqnarray}\label{eq:cond_laplace_sparse}
& & \varepsilon\times\frac{1}{2} \mathrm{e}^{- \sfrac{t_{\alpha_n,k}}{2}+
(1-\varepsilon)(\mu_2-\mu_1) } > C(\alpha,\beta)
\frac{\sqrt
{k}}{n}
\nonumber
\\[-8pt]\\[-8pt]
&&\quad  \Leftrightarrow \quad \varepsilon\times\frac{1}{2} \mathrm{e}^{- \sfrac
{t_{\alpha_n,k}}{2}+(\mu_2-\mu_1)}
\bigl(1-\mathrm{o}(1)\bigr) > C(\alpha,\beta)\frac
{\sqrt{k}}{n}, \nonumber
\end{eqnarray}
for some positive constant $C(\alpha,\beta)$, as $n\rightarrow
+\infty$. Thanks to \eqref{eq:control_t}, the inequality \eqref
{eq:cond_laplace_sparse} becomes\vspace*{-1pt}
\begin{eqnarray*}
\frac{1}{n^\delta}\times\frac{1}{n^\Delta}\times n^r >
n^{-\vfrac
{1+\Delta}{2}} & \quad \Leftrightarrow\quad & \delta+\Delta-r < \frac{1+\Delta}{2}
\\
& \quad \Leftrightarrow\quad & r > \delta+ \frac{\Delta}{2} - \frac{1}{2}.
\end{eqnarray*}
In practice, the smallest possible parameter $\Delta$ will provide the
less restrictive separation condition. In the same time, we have to
ensure that the condition \eqref{eq:condition_tau_t} is satisfied. It
follows from \eqref{eq:control_t} that
$t_{\alpha_n,k}/{2}\sim\Delta\log(n)$ as $n \rightarrow\infty$,
and \eqref{eq:condition_tau_t} holds for $n$ large enough as soon as
$\Delta> r$. Hence, choosing $\Delta= r+ r_0$ for some positive
$r_0$, we can remark that\vspace*{-1pt}
\[
r > \delta+ \frac{\Delta}{2} - \frac{1}{2}\quad  \Leftrightarrow\quad  r> 2(
\delta-1/2) + r_0,
\]
which is satisfied as soon as\vspace*{-1pt}
\[
r> 2(\delta-1/2),
\]
provided $r_0$ is small enough. This concludes the proof.

\begin{appendix}
\section*{Appendix: Lemmas for the upper-bound}\label{app}
\setcounter{equation}{0}%

\begin{lem}\label{Append:Lemma:quant}
Let $Y_1,\ldots,Y_n$ be $n$ random variables with a cumulative
distribution function $F$ and the order statistics are denoted
$Y_{(1)}\leq Y_{(2)}\leq\cdots\leq Y_{(n)}$. Let $\alpha\in\,]\,0,1[$ and
let $k\in\{1,\ldots,n\}$ such that $k > 2 \log (\frac
{2}{\alpha} )$. Let $c$ and $d$ be
two real numbers such that
%
\begin{equation}\label{Append:Defcd}
F(d) \vee\bigl(1-F(c)\bigr) \leq\frac{k}{n} \biggl[1-\sqrt{
\frac{2\log(\sfrac
{2}{\alpha})}{k}} \biggr].
\end{equation}
Then $\P(Y_{(n-k+1)}\geq c)\leq\alpha$ and $\P(Y_{(k)}\leq d)\leq
\alpha$.
\end{lem}

\begin{pf}
\begin{eqnarray*}
\P(Y_{(n-k+1)}\geq c) &=& \P \Biggl(\sum_{i=1}^n
\mathds{1}_{\{
Y_i\geq c\}}\geq k \Biggr)
\\
&=& \P \Biggl(\sum_{i=1}^n \bigl\{
\mathds{1}_{\{Y_i\geq c\}} - \bigl[1-F(c)\bigr]\bigr\} \geq k- n \bigl[1- F(c)
\bigr] \Biggr).
\end{eqnarray*}
According to condition \eqref{Append:Defcd},
\[
k- n \bigl[1-F(c)\bigr] \geq k \sqrt{\frac{2\log(\sfrac{2} {\alpha})}{k}} >0. %
\]
Using a Bernstein's inequality, we get
\[
\P(Y_{(n-k+1)}\geq c) \leq2\exp \biggl[-\frac{1} 2
\frac
{(k-n[1-F(c)])^2}{v + \sklfrac{1} {3} (k-n[1-F(c)])} \biggr] %
\]
with $v=\sum_{i=1}^n \E[(\mathds{1}_{\{Y_i\geq c\}} -
[1-F(c)])^2]=\sum_{i=1}^n \Var(\mathds{1}_{Y_i\geq
c})=nF(c)[1-F(c)]\leq n[1-F(c)]$. Thus,
$3v + k-n[1-F(c)]\leq2n[1-F(c)]+k\leq3 k - 2 k \sqrt{\frac{2\log
(\sfrac{2}{\alpha})}{k}}\leq3 k$. This implies that
\[
\P(Y_{(n-k+1)}\geq c)\leq2\exp \biggl[-\frac{3} 2
\frac
{(k-n[1-F(c)])^2}{3 k} \biggr]\leq2\exp \biggl[-\log \biggl(\frac{2}{\alpha}
\biggr) \biggr]=\alpha. %
\]
In the same way,
\begin{eqnarray*}
\P(Y_{(k)}\leq d) &=& \P \Biggl(\sum_{i=1}^n
\mathds{1}_{\{Y_i\geq
d\}}\leq n-k \Biggr)
\\
&=& \P \Biggl(\sum_{i=1}^n \bigl\{
\mathds{1}_{\{Y_i\geq d\}} - \bigl[1-F(d)\bigr]\bigr\} \leq n F(d)-k \Biggr).
\end{eqnarray*}
Since $n F(d)-k <0$ according to condition \eqref{Append:Defcd}, a
Bernstein's inequality implies that
\begin{eqnarray*}
\P(Y_{(k)}\leq d) &\leq&\P \Biggl(\Biggl\llvert \sum
_{i=1}^n \bigl\{\mathds{1}_{\{
Y_i\geq d\}} -
\bigl[1-F(d)\bigr]\bigr\}\Biggr\rrvert \geq k-n F(d) \Biggr)\\
&\leq&2\exp \biggl[-
\frac{1} 2 \frac{[nF(d)-k]^2}{v + \sklfrac{1} {3} [k-nF(d)]} \biggr] %
\end{eqnarray*}
with $v=\sum_{i=1}^n \E[(\mathds{1}_{\{Y_i\geq d\}} -
[1-F(d)])^2]=\sum_{i=1}^n \Var(Y_i\geq d)=nF(d)[1-F(d)]\leq n F(d)$. Thus,
$3v + k- nF(d)\leq2n F(d) + k \leq3k - 2 k \sqrt{\frac{2\log(\sfrac
{2}{\alpha})}{k}}\leq3 k$. This implies that
\[
\P(Y_{(k)}\leq d)\leq2\exp \biggl[-\frac{3} 2
\frac{[nF(d)-k]^2}{3
k} \biggr]\leq2\exp \biggl[-\log \biggl(\frac{2} \alpha
\biggr) \biggr]=\alpha. %
\]
\end{pf}

%
\begin{lem}\label{Lemma:MinBarPhi}
If $k\geq8 \log ({4}/{\alpha_n} )$ and $\frac{k} n \geq
n^{-\Delta}$ with $\Delta\in\,]\,0,1[$, then
\[
t_{\alpha_n,k} \leq2 \sqrt{2 \Delta\log(n)}.
\]
\end{lem}

\begin{pf}
\begin{eqnarray*}
\bar\Phi_G \biggl(\frac{t_{\alpha_n,k}}{2} \biggr) &=& \frac{k}{n}
\biggl[1 - \sqrt{\frac{2 \log(4/\alpha_n)}{k}} \biggr]
\\
&\leq& \frac{1}{2} \exp \biggl[- \frac{1} 2 \biggl(
\frac{t_{\alpha
_n,k}}{2} \biggr)^2 \biggr],
\end{eqnarray*}
thus
\[
\exp \biggl[\frac{1} 2 \biggl(\frac{t_{\alpha_n,k}}{2} \biggr)^2
\biggr] \leq\frac{1} 2 \biggl[1 - \sqrt{\frac{2 \log(4/\alpha
_n)}{k}}
\biggr]^{-1} n^{\Delta}.
\]
If $ k\geq8 \log ({4}/{\alpha_n } )$, then
\[
2 \biggl[1 - \sqrt{\frac{2 \log(4/\alpha_n)}{k}} \biggr] \geq1 %
\]
which leads to $t_{\alpha_n,k} \leq2 \sqrt{2\Delta\log(n)}$.
\end{pf}
\end{appendix}

\section*{Acknowledgements}
The authors would like to thank the associate editor and the two
referees for their constructive remarks that have helped to improve the paper.
They also acknowledge the support of the French Agence Nationale de la
Recherche (ANR), under grant MixStatSeq (ANR-13-JS01-0001-01).



\printhistory
\end{document}